\documentclass[]{article}

\author{Rose Wagstaffe}

\usepackage{amssymb,amsmath,amsopn,amsthm,graphicx,makeidx, mathrsfs, enumerate, enumitem}
\input xy
\xyoption{all}
\usepackage{tikz}
\usetikzlibrary{matrix}
\usetikzlibrary{arrows.meta}
\usepackage[a4paper, top={1.75in}, bottom={2in}]{geometry}
\usepackage{placeins}
\usepackage{helvet}	
\usepackage{tikz-cd}
\usepackage{verbatim}
\usepackage{stmaryrd}

\theoremstyle{definition}	
\newtheorem{definition}{Definition}[section]

\theoremstyle{plain}
\newtheorem{lemma}[definition]{Lemma}
\newtheorem{prop}[definition]{Proposition}
\newtheorem{theorem}[definition]{Theorem}
\newtheorem{cor}[definition]{Corollary}

\theoremstyle{remark}
\newtheorem{remark}[definition]{Remark}

\newtheorem{notation}[definition]{Notation}
\newtheorem{assumption}[definition]{Assumption}
\theoremstyle{plain}
\newtheorem{example}[definition]{Example}

\theoremstyle{definition}

\theoremstyle{remark}

\renewenvironment{proof}{\noindent {\bf{Proof.}}}{\hspace*{3mm}{$\Box$}{\vspace{9pt}}}

\begin{document}
	
%	\fontfamily{helvet}\selectfont
	
	\setcounter{page}{1}
	
	\title{A monoidal analogue of the 2-category anti-equivalence between $\mathbb{ABEX}$ and $\mathbb{DEF}$}
	
	\date{}
	
	\maketitle
	
 \begin{abstract}
 	We prove that the 2-category of skeletally small abelian categories with exact monoidal structures is anti-equivalent to the 2-category of fp-hom-closed definable additive categories satisfying an exactness criterion.
 	
 	For a fixed finitely accessible category $\mathcal{C}$ with products and a monoidal structure satisfying the appropriate assumptions, we provide bijections between the fp-hom-closed definable subcategories of $\mathcal{C}$, the Serre tensor-ideals of $\mathcal{C}^{\mathrm{fp}}\hbox{-}\mathrm{mod}$ and the closed subsets of a Ziegler-type topology. 
 	
 	For a skeletally small preadditive category $\mathcal{A}$ with an additive, symmetric, rigid monoidal structure we show that elementary duality induces a bijection between the fp-hom-closed definable subcategories of $\mathrm{Mod}\hbox{-}\mathcal{A}$ and the definable tensor-ideals of $\mathcal{A}\hbox{-}\mathrm{Mod}$. 
 \end{abstract}	
	
\section{Introduction}

In \cite{PR10}, Prest and Rajani define an anti-equivalence between the 2-category $\mathbb{ABEX}$ with objects given by skeletally small abelian categories and morphisms given by additive exact functors and the 2-category $\mathbb{DEF}$ with objects given by definable categories and morphisms given by additive functors which commute with direct products and direct limits. In both cases the 2-morphisms are given by natural transformations. Under this equivalence, a finitely accessible category with products $\mathcal{C}$, which by definition is definable, corresponds to the module category $\mathcal{C}^{\mathrm{fp}}\hbox{-}\mathrm{mod}$ of finitely presented additive functors from the full subcategory of $\mathcal{C}$ given by the finitely presented objects, $\mathcal{C}^{\mathrm{fp}}$, to the category of abelian groups, $\mathbf{Ab}$. In the case that $\mathcal{C}^{\mathrm{fp}}$ has a monoidal structure, we may induce a monoidal structure on $\mathcal{C}^{\mathrm{fp}}\hbox{-}\mathrm{mod}$ via Day convolution product (see \cite{D70}). What's more, the definable subcategories of $\mathcal{C}$ correspond to the Serre subcategories of $\mathcal{C}^{\mathrm{fp}}\hbox{-}\mathrm{mod}$ (e.g. see \cite{BRB}, Theorem 12.4.1 and Corollary 12.4.2) and the equivalence maps a definable subcategory $\mathcal{D} \subseteq \mathcal{C}$ to the functor category $\mathrm{fun}(\mathcal{D}) \simeq \mathcal{C}^{\mathrm{fp}}\hbox{-}\mathrm{mod}/\mathsf{S}$ where $\mathsf{S}$ is the Serre subcategory corresponding (by annihilation) to $\mathcal{D}$ (e.g. see \cite{P11}, Theorem 12.10). When the Serre subcategory $\mathsf{S}$ is a Serre tensor-ideal, the functor category $\mathrm{fun}(\mathcal{D}) \simeq \mathcal{C}^{\mathrm{fp}}\hbox{-}\mathrm{mod}/\mathsf{S}$ has a monoidal structure (see Definition \ref{remark strict}). Therefore, it is interesting to ask: Which definable subcategories of $\mathcal{C}$ correspond to Serre tensor-ideals of $\mathcal{C}^{\mathrm{fp}}\hbox{-}\mathrm{mod}$? The answer is given in Theorem \ref{thm}. More generally, we define 2-categories $\mathbb{ABEX}^{\otimes}$ and $\mathbb{DEF}^{\otimes}$ and give a 2-category anti-equivalence which can be viewed as a monoidal analogue of the 2-category anti-equivalence between $\mathbb{ABEX}$ and $\mathbb{DEF}$ (Section \ref{section 2-cat equiv}). 

In order for the 2-category anti-equivalence between $\mathbb{ABEX}^{\otimes}$ and $\mathbb{DEF}^{\otimes}$ to hold, we must take the objects of $\mathbb{DEF}^{\otimes}$ to be the definable subcategories $\mathcal{D} \subseteq \mathcal{C}$ which are fp-hom-closed and satisfy an exactness criterion. Here $\mathcal{C}$ is a finitely accessible category with products and an additive, symmetric, closed monoidal structure such that the subcategory of finitely presented objects, $\mathcal{C}^{\mathrm{fp}}$, forms a symmetric monoidal subcategory. By the original anti-equivalence, $\mathcal{D}$ and $\mathrm{Ex}(\mathrm{fun}(\mathcal{D}),\mathbf{Ab})$ are equivalent definable categories. However, the exactness criterion is necessary to ensure that the fp-hom-closed property is preserved (Theorem \ref{TFAE SE tensor ideal} and Proposition \ref{Lem exactness if}).  

In practice, many fp-hom-closed definable subcategories do not satisfy the exactness criterion. Indeed, the exactness crition for $\mathcal{D}$ implies that the monoidal structure on $\mathrm{fun}(\mathcal{D})$ is exact (Theorem \ref{prop exactness only if}), when in general this monoidal structure is only right exact. In Section \ref{without exactness}, we discuss the relationship between definability and monoidal structures for fixed $\mathcal{C}$ without assuming the exactness criterion. We define a `coarser-version' of the Ziegler spectrum (Section \ref{section ziegler}) and provide bijections between the fp-hom-closed definable subcategories of $\mathcal{C}$, the Serre tensor-ideals of $\mathcal{C}^{\mathrm{fp}}\hbox{-}\mathrm{mod}$ and the closed subsets of this Ziegler-type topology. We also consider what can be said under the additional assumption that $\mathcal{C}^{\mathrm{fp}}$ is rigid monoidal (Section \ref{section rigid}). Here a definable subcategory is fp-hom-closed if and only if it is a tensor-ideal (Corollary \ref{Cfp rigid case}). Furthermore, given a skeletally small preadditive category $\mathcal{A}$ with an additive, symmetric, rigid monoidal structure, we provide a bijection between the fp-hom-closed definable subcategories of $\mathrm{Mod}\hbox{-}\mathcal{A}$ and the definable tensor-ideals of $\mathcal{A}\hbox{-}\mathrm{Mod}$ (Section \ref{section duality}). 

Finally, in Section \ref{section examples} we will consider some examples. First we explore the examples given by the tensor product of $R$-modules for a commutative ring $R$ and secondly we give an example where $\mathcal{C}^{\mathrm{fp}}$ is rigid.

\subsection{Acknowledgements}

I would like to express my great appreciation to Prof. Mike Prest, my PhD supervisor, for his patience and guidance. I would also like to thank EPSRC and the University of Manchester for providing the funding for my PhD. 

\section{Preliminaries}

\subsection{Day convolution product} \label{Day convolution product}

Given a finitely accessible category, $\mathcal{C}$, with products and a monoidal structure, we will use Day convolution product to induce a monoidal structure on the associated functor category. 

\begin{theorem} \label{Day} (\cite{D70}, Theorem 3.3 and Theorem 3.6)
	Given a complete and cocomplete closed symmetric monoidal category $V$, and a small (symmetric) monoidal $V$-enriched category $\mathscr{C}$, the category of $V$-enriched functors from $\mathscr{C}$ to $V$, $V[\mathscr{C},V]$, is a monoidal category admitting a (symmetric) closed monoidal structure.
\end{theorem}   

\begin{notation}
	Throughout our $V$ (as above) will be the category of abelian groups, $\mathbf{Ab}$. All our categories will be preadditive and all our functors will be additive. For a preadditive category $\mathscr{C}$ we denote by $\mathscr{C}(A,B)$ the abelian group of all morphisms in $\mathscr{C}$ from $A$ to $B$. When the category is clear from context we will simply write $(A,B)$. In addition, given preadditive categories $\mathcal{A}$ and $\mathcal{B}$, we will denote by $(\mathcal{A},\mathcal{B})$ the functor category of all additive functors from $\mathcal{A}$ to $\mathcal{B}$. The functor category $(\mathcal{A},\mathbf{Ab})$ will be denoted by $\mathcal{A}\hbox{-}\mathrm{Mod}$ and the subcategory of all finitely presented objects with be denoted by $\mathcal{A}\hbox{-}\mathrm{mod}:=(\mathcal{A}\hbox{-}\mathrm{Mod})^{\mathrm{fp}}$. Similarly, we denote by $\mathrm{Mod}\hbox{-}\mathcal{A}$ and $\mathrm{mod}\hbox{-}\mathcal{A}$ the categories $(\mathcal{A}^{\mathrm{op}},\mathbf{Ab})$ and $(\mathcal{A}^{\mathrm{op}},\mathbf{Ab})^{\mathrm{fp}}$ respectively. 
\end{notation}

Given a symmetric monoidal structure $(\otimes, 1)$ on a small additive category $\mathcal{A}$, we may refer to the Day convolution product on the functor category 
$\mathcal{A}\hbox{-}\mathrm{Mod}:=(\mathcal{A},\mathbf{Ab})$ as the `induced monoidal structure' or `induced tensor product' and denote the tensor product functor by $\otimes$. By Theorem \ref{Day}, the induced monoidal structure on $\mathcal{A}\hbox{-}\mathrm{Mod}$ is closed, that is, for every $X \in \mathcal{A}\hbox{-}\mathrm{Mod}$, $X \otimes -:\mathcal{A}\hbox{-}\mathrm{Mod} \to \mathcal{A}\hbox{-}\mathrm{Mod}$ has a right adjoint functor which we will denote by $\mathrm{hom}(X,-):\mathcal{A}\hbox{-}\mathrm{Mod} \to \mathcal{A}\hbox{-}\mathrm{Mod}$ and call the \textbf{internal hom-functor}.  

Since for each $F \in \mathcal{A}\hbox{-}\mathrm{Mod}$, $F \otimes -$ is a left adjoint, it is right exact and commutes with direct limits. Furthermore, by definition of Day convolution product, given representable functors $(A,-)$ and $(B,-)$ in $\mathcal{A} \hbox{-}\mathrm{Mod}$, we have $(A,-) \otimes (B,-) \cong (A \otimes B,-)$. Thus, by right exactness, if $F \in \mathcal{A}\hbox{-}\mathrm{mod}$ has presentation $(B,-) \xrightarrow{(f,-)} (A,-) \to F \to 0$,  with $f:A \to B$, then $(C,-) \otimes F$ has presentation $(C \otimes B,-) \xrightarrow{(C \otimes f,-)} (C \otimes A,-) \to (C,-) \otimes F \to 0$. 

\begin{notation} \label{notation Ff}
	Given an additive (skeletally) small category $\mathcal{A}$ every finitely presented module $F \in \mathcal{A}\hbox{-}\mathrm{mod}$ has a presentation of the form \[ (B,-) \xrightarrow{(f,-)} (A,-) \xrightarrow{\pi_f} F \to 0, \] with $f:A \to B$ in $\mathcal{A}$. We will denote such a functor by $F_f$. 
\end{notation}

By the above we have $(C,-) \otimes F_f=F_{C \otimes f}$. More generally, $F_f \otimes F_g=F_{(f \otimes U, A \otimes g)}$, where $f:A \to B$ and $g:U \to V$ and $(f \otimes U, A \otimes g):A \otimes U \to (B \otimes U) \oplus (A \otimes V)$ is the canonical map. 

Thus, Day convolution product restricts to a monoidal structure on the category of finitely presented additive functors $\mathcal{A}\hbox{-}\mathrm{mod}$, which we may also refer to as the `induced monoidal structure' or `induced tensor product'. This is exactly the tensor product given in (\cite{P18}, Section 13.3) with $\mathcal{A}=R\hbox{-}\mathrm{mod}$. Here we avoid the notation $(R\hbox{-}\mathrm{mod})\hbox{-}\mathrm{mod}$ in favour of $(R\hbox{-}\mathrm{mod},\mathbf{Ab})^{\mathrm{fp}}$.

\subsection{Rigid monoidal categories}

In this section we will outline the definition of a rigid monoidal category.

\begin{definition}
Let $\mathscr{C}$ be a symmetric monoidal category. $C^{\vee} \in \mathscr{C}$ is \textbf{dual} to $C \in \mathscr{C}$ if there exist morphisms $\eta:1 \to C^{\vee} \otimes C$ and $\epsilon:C \otimes C^{\vee} \to 1$ such that $(C^{\vee} \otimes \epsilon) \circ (\eta \otimes C^{\vee})= \mathrm{Id}_{C^{\vee}}$ and $(\epsilon \otimes C) \circ (C \otimes \eta)= \mathrm{Id}_C$.

A closed symmetric monoidal category $\mathscr{C}$ is said to be \textbf{rigid} if every object of $\mathscr{C}$ has a dual.
\end{definition} 

An important consequence of the existence of dual objects is the following.

\begin{prop} (e.g. \cite{EGNO15}, Proposition 1.10.9)
	Let $\mathscr{C}$ be a symmetric monoidal category and suppose $C \in \mathscr{C}$ is rigid. Then $C^{\vee} \otimes -$ is both left and right adjoint to $C \otimes -$.  
\end{prop}

\begin{cor}
	Let $\mathscr{C}$ be a closed symmetric monoidal category and suppose $C^{\vee}$ is dual to $C$ in $\mathscr{C}$. There exists a natural isomorphism $\mathrm{hom}(C,-) \cong C^{\vee} \otimes -$.
\end{cor}

\begin{cor}
	Let $\mathcal{A}$ be an abelian category with a closed symmetric monoidal structure and suppose $C \in \mathcal{A}$ has a dual. Then $C \otimes -:\mathcal{A} \to \mathcal{A}$ is exact.
\end{cor}

\begin{definition}
	Let $\mathscr{C}$ be a rigid symmetric monoidal category. Given any morphism, $f:A \to B$ in $\mathscr{C}$, there exists a dual morphism, $f^{\vee}:B^{\vee} \to A^{\vee}$, in $\mathscr{C}$ given by the composition 
	\[ B^{\vee} \xrightarrow{\eta \otimes B^{\vee}} A^{\vee} \otimes A \otimes B^{\vee} \xrightarrow{A^{\vee} \otimes f \otimes A^{\vee}} A^{\vee} \otimes B \otimes B^{\vee} \xrightarrow{A^{\vee} \otimes \epsilon} A^{\vee}.\]
\end{definition}

\subsection{Purity in finitely accessible categories} \label{Section lfp background}

The results in this section will be stated without proof and we direct the reader to \cite{BRB}, \cite{P11} and \cite{PR10} for more details. Throughout the paper, we use \cite{BRB} and \cite{P11} as convenient secondary sources.

Let us recall the definition of a finitely accessible category.

\begin{definition}
	A category $\mathcal{C}$ is said to be \textbf{finitely accessible}, if it has direct limits and there exists a set, $\mathscr{G}$, of finitely presentable objects of $\mathcal{C}$ such that for every $X \in \mathcal{C}$, we can write $X$ as a direct limit of objects of $\mathscr{G}$. That is, $X=\underrightarrow{\mathrm{lim}}_{i \in I} X_i$ where $I$ is some directed indexing set and each $X_i \in \mathscr{G}$. Note that in this case, the full subcategory of finitely presentable objects of $\mathcal{C}$, denoted by $\mathcal{C}^{\mathrm{fp}}$, is skeletally small and we can take $\mathscr{G}$ to consist of a representative of each isomorphism class of $\mathcal{C}^{\mathrm{fp}}$. For the purposes of this paper we will take `finitely accessible' to mean additive and finitely accessible. 
\end{definition}

Next we give the definition of a definable subcategory of a finitely accessible category with products.

\begin{definition}
	Let $\mathcal{C}$ be a finitely accessible category with products. A full subcategory $\mathcal{D}\subseteq \mathcal{C}$ is said to be \textbf{definable} if it is closed in $\mathcal{C}$ under products, direct limits and pure subobjects.
	
	A \textbf{definable category} is a definable subcategory of some finitely accessible category with products.
\end{definition}

We can use the definable subcategories of a finitely accessible category with product to define a topology called the Ziegler spectrum.

\begin{definition}
Let $\mathcal{C}$ be finitely accessible with products. A monomorphism $m:X \to Y$ in $\mathcal{C}$ is said to be a \textbf{pure monomorphism} if for every $f:A \to B$ in $\mathcal{C}^{\mathrm{fp}}$ and for all morphisms $h:A \to X$ and $h':B \to Y$ such that $h' \circ f=m \circ h$ there exist some $k:B \to X$ such that $k \circ f=h$. 
	\begin{tikzpicture}
\matrix (m) [matrix of math nodes,row sep=4em,column sep=4em,minimum width=2em]
{
	A & B   \\
	X & Y \\};
\path[-stealth]
(m-1-1) edge node [above] {$f$} (m-1-2)
(m-1-1) edge node [left] {$h$} (m-2-1)
(m-1-2) edge node [left] {$h'$} (m-2-2)
(m-2-1) edge node [below] {$m$} (m-2-2)
(m-1-2) edge [dashed] node [below] {$k$} (m-2-1);
\end{tikzpicture}
\end{definition} 

\begin{remark}
	If $\mathcal{C}$ is locally finitely presented (that is finitely accessible, complete and cocomplete), pure monomorphisms can be characterised as those monomorphism $m:X \to Y$ which fit into an exact sequence \[0 \to X \xrightarrow{m} Y \xrightarrow{p} Z \to 0\] such that for every $A \in \mathcal{C}^{\mathrm{fp}}$, \[0 \to (A,X) \xrightarrow{(A,m)} (A,Y) \xrightarrow{(A,p)} (A,Z) \to 0\] is exact in $\mathbf{Ab}$ (see \cite{P11}, Theorem 5.2).
\end{remark}

\begin{definition}
	We say that an object $E \in \mathcal{C}$ is \textbf{pure-injective} if it is injective over pure monomorphisms, that is for every pure monomorphism $m:X \to Y$ in $\mathcal{C}$ and any morphism $k:X \to E$ there exists some $h:Y \to E$ such that $k=h \circ m$.
\end{definition} 

In fact, each finitely accessible category with products has, up to isomorphism, a set of indecomposable pure-injective objects and each definable subcategory is generated as such by its indecomposable pure-injectives. They form the underlying set of a topological space called the Ziegler spectrum.

\begin{definition}
	We define the \textbf{Ziegler spectrum} of $\mathcal{C}$, denoted $\mathrm{Zg}(\mathcal{C})$, to have underlying set given by the set of isomorphism classes of indecomposable pure-injectives in $\mathcal{C}$, denoted $\mathrm{pinj}_{\mathcal{C}}$, and closed subsets given by \[ \{[X] \in \mathrm{pinj}_{\mathcal{C}}: X \in \mathcal{D} \} \] where $[X]$ denotes the isomorphism classes of the indecomposable pure-injective $X$ and $\mathcal{D}$ runs through the definable subcategories of $\mathcal{C}$.
\end{definition}

\begin{prop} (\cite{P11}, Theorem 14.1)
	Let $\mathcal{C}$ be a finitely accessible category with products. The closed subsets described above define a topology on $\mathrm{pinj}_{\mathcal{C}}$. 
\end{prop}

Next we give another way to find the definable subcategories of $\mathcal{C}$.

\begin{notation} Let $\mathcal{C}$ be a finitely accessible category. For $F \in \mathcal{C}^{\mathrm{fp}}\hbox{-}\mathrm{mod}$, denote by $\overrightarrow{F}:\mathcal{C} \to \mathbf{Ab}$ the unique extension of $F$ which commutes with direct limits (e.g. see \cite{BRB}, Proposition 10.2.41). 
\end{notation}

If $\mathcal{C}$ is additive finitely accessible with products, then a subcategory, $\mathcal{D} \subseteq \mathcal{C}$ is definable if and only if there is a collection of finitely presented functors $\mathcal{Y} \subseteq \mathcal{C}^{\mathrm{fp}}\hbox{-}\mathrm{mod}$ such that $X \in \mathcal{D}$ if and only if $\overrightarrow{F}(X)=0$ for all $F \in \mathcal{Y}$. Furthermore, if $\mathcal{D} \subseteq \mathcal{C}$ is definable then the set $S=\{F \in \mathcal{C}^{\mathrm{fp}}\hbox{-}\mathrm{mod}:\overrightarrow{F}(X)=0,~\forall X \in \mathcal{D}\}$ is a Serre subcategory. In fact, we have Theorem \ref{loc fp def Serre} below.

\begin{theorem} \label{loc fp def Serre} (\cite{P11}, Theorem 14.2)
	Let $\mathcal{C}$ be an additive finitely accessible category with products. There is a natural bijection between:
	\begin{enumerate}[label=(\roman*)]
		\item  the definable subcategories of $\mathcal{C}$,
		
		\item the Serre subcategories of $\mathcal{C}^{\mathrm{fp}}\hbox{-}\mathrm{mod}$,
		
		\item the closed subsets of the Ziegler Spectrum $\mathrm{Zg}(\mathcal{C})$.
	\end{enumerate}
	  
\end{theorem}

Next let us describe elementary duality. Let $\mathcal{A}$ be a skeletally small preadditive category. First we define the tensor product of $\mathcal{A}$-modules, a generalisation of tensor product over a ring. 

\begin{definition} \label{def tensor over A} (see for example \cite{P11}, Section 3)
	The \textbf{tensor product of} $\mathcal{A}$\textbf{-modules} is given by a functor $- \otimes_{\mathcal{A}}-:\mathrm{Mod}\hbox{-}\mathcal{A} \times \mathcal{A}\hbox{-}\mathrm{Mod} \to \mathbf{Ab}$ determined on objects (up to isomorphism) by the following two assertions. For every $M \in \mathrm{Mod}\hbox{-}\mathcal{A}$,
	
	\begin{enumerate}[label=(\roman*)]
		\item  $M \otimes_{\mathcal{A}} (A,-)\cong M(A)$ for every $A \in \mathcal{A}$,
		
		\item  $M \otimes_{\mathcal{A}} -$ is right exact. 
	\end{enumerate}

The functor is defined on morphisms in the obvious way.
\end{definition}

We can now define a duality of functor categories as follows. 

\begin{theorem} \label{thm background duality} (\cite{P11}, Theorem 4.5)
	There is a duality $\delta:(\mathrm{mod}\hbox{-}\mathcal{A},\mathbf{Ab})^{\mathrm{fp}}\to (\mathcal{A}\hbox{-}\mathrm{mod},\mathbf{Ab})^{\mathrm{fp}} $ given on objects by mapping $F_f:\mathrm{mod}\hbox{-}\mathcal{A} \to \mathbf{Ab}$, where $f:A \to B$ in $\mathrm{mod}\hbox{-}\mathcal{A}$, to $\delta F: \mathcal{A} \hbox{-}\mathrm{mod} \to \mathbf{Ab}$ where $\delta F$ has copresentation \[ 0 \to \delta F \to A \otimes_{\mathcal{A}} - \xrightarrow{f \otimes_{\mathcal{A}} -} B \otimes_{\mathcal{A}} -.\]   
\end{theorem}

Next we note that $\delta$ induces a bijection between definable subcategories. 

\begin{prop} \label{prop background def dual} (\cite{P11}, Theorem 8.1)
	The duality, $\delta$, of Theorem \ref{thm background duality} maps Serre subcategories of $(\mathrm{mod}\hbox{-}\mathcal{A},\mathbf{Ab})^{\mathrm{fp}}$ to Serre subcategories of $(\mathcal{A}\hbox{-}\mathrm{mod},\mathbf{Ab})^{\mathrm{fp}}$ and therefore induces a bijection between the definable subcategories of $\mathrm{Mod}\hbox{-}\mathcal{A}$ and those of $\mathcal{A}\hbox{-}\mathrm{Mod}$.
\end{prop}

\begin{notation}
	Given a Serre subcategory $\mathsf{S} \subseteq (\mathrm{mod}\hbox{-}\mathcal{A},\mathbf{Ab})^{\mathrm{fp}}$ we will denote the dual Serre subcategory by $\delta \mathsf{S}$, that is $\delta \mathsf{S}=\{\delta F: F \in \mathsf{S}\} \subseteq (\mathcal{A}\hbox{-}\mathrm{mod},\mathbf{Ab})^{\mathrm{fp}}$.
	
	Similarly, given a definable subcategory $\mathcal{D} \subseteq \mathrm{Mod}\hbox{-}\mathcal{A}$ we will denote the dual definable subcategory, associated to $\delta \mathsf{S}$ by annihilation, by $\delta \mathcal{D} \subseteq \mathcal{A}\hbox{-}\mathrm{Mod}$. 
	
	We will also use this $\delta$ notation for the inverse map. That is, if $\mathsf{S} \subseteq (\mathcal{A}\hbox{-}\mathrm{mod},\mathbf{Ab})^{\mathrm{fp}}$ is a Serre subcategory $\delta \mathsf{S} \subseteq (\mathrm{mod}\hbox{-}\mathcal{A},\mathbf{Ab})^{\mathrm{fp}}$ is the dual Serre subcategory and similarly for definable subcategories.
\end{notation}

Below we give two key properties of the 2-category anti-equivalence between $\mathbb{ABEX}$ and $\mathbb{DEF}$. See \cite{PR10} for full details.

\begin{definition}
	Let  $\mathbb{DEF}$ denote the 2-category with objects given by definable categories, morphisms given by additive functors which preserve direct product and direct limits and 2-morphisms given by natural transformations. 
	
	Let $\mathbb{ABEX}$ denote the 2-category with objects given by skeletally small abelian categories, morphisms given by additive exact functors and 2-morphisms given by natural transformations. 
\end{definition}

\begin{theorem} \label{Thm ABEX DEF} (\cite{PR10}, Theorem 2.3) 
	There exists a 2-category anti-equivalence between $\mathbb{ABEX}$ and  $\mathbb{DEF}$ given on objects by $\mathscr{A} \mapsto \mathrm{Ex}(\mathscr{A},\mathbf{Ab})$ and $\mathcal{D} \mapsto \mathrm{fun}(\mathcal{D}):=(\mathcal{D},\mathbf{Ab})^{\to \Pi}$, where $\mathrm{Ex}(\mathscr{A},\mathbf{Ab})$ is the category of exact functors from $\mathscr{A}$ to the category of abelian groups and $(\mathcal{D},\mathbf{Ab})^{\to \Pi}$ is the category of additive functors from $\mathcal{D}$ to the category of abelian groups which commute with direct products and direct limits. 
\end{theorem}

On morphisms the equivalence works in both directions by mapping an appropriate functor, say $F$, to precomposition by $F$, $- \circ F$, and on 2-morphisms it works in the obvious way. 

\begin{theorem} \label{Thm fun(D)} (\cite{P11},Theorem 12.10)
	Given a definable subcategory $\mathcal{D}$ of a finitely accessible category $\mathcal{C}$ with products, $\mathrm{fun}(\mathcal{D}) \simeq \mathcal{C}^{\mathrm{fp}}\hbox{-}\mathrm{mod}/\mathsf{S}$ where $\mathsf{S} \subseteq \mathcal{C}^{\mathrm{fp}}\hbox{-}\mathrm{mod}$ is the Serre subcategory corresponding to $\mathcal{D}$ (as in Theorem \ref{loc fp def Serre}).
\end{theorem}

Given a finitely presented functor $F \in \mathcal{C}^{\mathrm{fp}}\hbox{-}\mathrm{mod}$, the restriction to $\mathcal{D}$ of its extension along direct limits, $(\overrightarrow{F})|_{\mathcal{D}}:\mathcal{D} \to \mathbf{Ab}$, commutes with direct product and direct limits and therefore is an object of $\mathrm{fun}(\mathcal{D})$. Let $\mathsf{S} \subseteq \mathcal{C}^{\mathrm{fp}}$ be the Serre subcategory corresponding to $\mathcal{D}$ and recall that $\mathcal{C}^{\mathrm{fp}}\hbox{-}\mathrm{mod}/\mathsf{S}$ is given by formally inverting the morphisms in $\Sigma_{\mathsf{S}}=\{\alpha \in \mathrm{morph}(\mathcal{C}^{\mathrm{fp}}): \mathrm{ker}(\alpha),~\mathrm{coker}(\alpha) \in \mathsf{S}\}$. Since every morphism in $\Sigma_{\mathsf{S}}$ is an isomorphism when evaluated at any $D \in \mathcal{D}$, by the universal property of the localisation, the functor $(\overrightarrow{-})|_{\mathcal{D}}:\mathcal{C}^{\mathrm{fp}}\hbox{-}\mathrm{mod} \to \mathrm{fun}(\mathcal{D})$ factors via the localisation $\mathcal{C}^{\mathrm{fp}}\hbox{-}\mathrm{mod}/\mathsf{S}$. The equivalence in Theorem \ref{Thm fun(D)} is given by the exact functor $(\overrightarrow{-})|_{\mathcal{D}}:\mathcal{C}^{\mathrm{fp}}\hbox{-}\mathrm{mod}/\mathsf{S} \to \mathrm{fun}(\mathcal{D})$ induced by this factorisation.

\subsection{The 2-categories $\mathbb{DEF}^{\otimes}$ and $\mathbb{ABEX}^{\otimes}$}

In this section we define the 2-categories $\mathbb{ABEX}^{\otimes}$ and $\mathbb{DEF}^{\otimes}$.

\begin{notation}
	Every monoidal category is monoidally equivalent to a strict monoidal category (\cite{ML13}, Section XI, Subsection 3, Theorem 1). Therefore we are safe to suppress all unitors and associators, treating them as identities.  
\end{notation}

\begin{definition}
	We will say that a functor $F:\mathscr{A} \to \mathscr{B}$ between monoidal categories $(\mathscr{A}, \otimes, 1_{\mathscr{A}})$ and $(\mathscr{B},\otimes', 1_{\mathscr{B}})$ is \textbf{monoidal} if there exists an isomorphism in $\mathscr{B}$, $\epsilon:1_{\mathscr{B}} \to F(1_{\mathscr{A}})$ and a natural isomorphism $\mu:(\otimes' \circ F \times F) \to (F \circ \otimes)$ satisfying the associativity condition $\mu_{X \otimes Y, Z} \circ (\mu_{X, Y} \otimes' F(Z))=\mu_{X, Y \otimes Z} \circ (F(X) \otimes' \mu_{Y, Z})$ and the unitality conditions, $\mu_{1_{\mathscr{A}}, X} \circ (\epsilon \otimes' F(X))=id_{1_{\mathscr{B}} \otimes' F(X)}$ and $\mu_{X, 1_{\mathscr{A}}} \circ (F(X) \otimes' \epsilon)=id_{F(X) \otimes' 1_{\mathscr{B}}}$.
\end{definition}

\begin{definition}
	Let $\mathbb{ABEX}^{\otimes}$ denote the 2-category with objects given by skeletally small abelian categories equipped with an additive symmetric monoidal structure which is exact in each variable, morphisms given by additive exact monoidal functors and 2-morphisms given by natural transformations. 
\end{definition}

\begin{notation}
	Given a category $\mathscr{C}$ and morphisms $f:A \to B$ and $k:A \to C$ in $\mathscr{C}$, we will write $f | k$ if there exists some morphism $k':B \to C$ in $\mathscr{C}$ such that $k=k' \circ f$.
\end{notation}

\begin{definition}
Let $\mathcal{C}$ be a finitely accessible category with products and an additive symmetric closed monoidal structure such that $\mathcal{C}^{\mathrm{fp}}$ is a monoidal subcategory. 

We say that a definable subcategory $\mathcal{D} \subseteq \mathcal{C}$ is \textbf{fp-hom-closed} if for every $A \in \mathcal{C}^{\mathrm{fp}}$ and $X \in \mathcal{D}$, $\mathrm{hom}(A,X) \in \mathcal{D}$, where $\mathrm{hom}$ denotes the internal hom-functor. 

We say that a definable subcategory $\mathcal{D} \subseteq \mathcal{C}$ satisfies the \textbf{exactness criterion} if given morphisms $f:A \to B$ and $g:U \to V$ in $\mathcal{C}^{\mathrm{fp}}$ and a morphism $h:A \otimes U \to X$ in $\mathcal{C}$ where $X \in \mathcal{D}$, if $(f \otimes U) | h$ and $(A \otimes g)|h$ then $(f \otimes g)|h$.   
\end{definition}

\begin{definition}
	We define the 2-category $\mathbb{DEF}^{\otimes}$ as follows. Let the objects of $\mathbb{DEF}^{\otimes}$ be given by the triples $(\mathcal{D}, \mathcal{C}, \otimes)$ where $\mathcal{C}$ is a finitely accessible category with products, $\otimes$ is an additive symmetric closed monoidal structure on $\mathcal{C}$ such that $\mathcal{C}^{\mathrm{fp}}$ is a monoidal subcategory and $\mathcal{D}$ is an fp-hom-closed definable subcategory of $\mathcal{C}$ satisfying the exactness criterion. Let the morphisms in $\mathbb{DEF}^{\otimes}$ be given by the additive functors $I: \mathcal{D} \to \mathcal{D'}$ which commute with direct product and direct limits and such that the induced functor $I_0:\mathrm{fun}(\mathcal{D'}) \to \mathrm{fun}(\mathcal{D})$ (see \cite{PR10}, Theorem 2.3) is monoidal.
\end{definition}

\begin{remark}
	 Notice that there exist forgetful 2-functors $\mathscr{F}:\mathbb{ABEX}^{\otimes} \to \mathbb{ABEX}$ and $\mathscr{F}:\mathbb{DEF}^{\otimes} \to \mathbb{DEF}$ which forget the monoidal structure. 
\end{remark}

\section{The 2-category anti-equivalence} \label{section 2-cat equiv}

In this section we prove the main theorem of the paper.

\begin{theorem} \label{Thm 2-cat eq}
	There exists a 2-category anti-equivalence between $\mathbb{ABEX}^{\otimes}$ and $\mathbb{DEF}^{\otimes}$ given on objects by $\mathscr{A} \mapsto (\mathrm{Ex}(\mathscr{A},\mathbf{Ab}),\mathscr{A}\hbox{-}\mathrm{Mod}, \otimes)$ where the monoidal structure, $\otimes$, on $\mathscr{A}\hbox{-}\mathrm{Mod}$ is induced by the monoidal structure on $\mathscr{A}$ via Day convolution product. Conversely, the anti-equivalence maps an object $(\mathcal{D}, \mathcal{C}, \otimes)$ in $\mathbb{DEF}^{\otimes}$ to the skeletally small abelian category $\mathrm{fun}(\mathcal{D})=(\mathcal{D},\mathbf{Ab})^{\Pi \to}$ with monoidal structure induced by Day convolution product on $\mathcal{C}^{\mathrm{fp}}\hbox{-}\mathrm{mod}$ (see Definition \ref{remark strict}). 
\end{theorem}

We prove Theorem \ref{Thm 2-cat eq} in several parts. 

\subsection{The 2-functor $\theta:\mathbb{DEF}^{\otimes} \to \mathbb{ABEX}^{\otimes}$}

First let us define a 2-functor $\theta:\mathbb{DEF}^{\otimes} \to \mathbb{ABEX}^{\otimes}$. On objects, $\theta$ maps $(\mathcal{D},\mathcal{C},\otimes)$ in $\mathbb{DEF}^{\otimes}$ to $\mathrm{fun}(\mathcal{D})$. In Theorem \ref{thm} we show that the Serre subcategory of $\mathcal{C}^{\mathrm{fp}}\hbox{-}\mathrm{mod}$ corresponding to $\mathcal{D}$ is a Serre tensor-ideal. We use this in Definition \ref{remark strict} to define an additive symmetric monoidal structure on $\mathrm{fun}(\mathcal{D})$. We then show that the monoidal structure is exact in each variable in Proposition \ref{prop exactness only if}.

\begin{assumption} \label{Rmk C}
	Let $\mathcal{C}$ be an additive finitely accessible category with products. Suppose further that $\mathcal{C}$ has an additive closed symmetric monoidal structure such that $\mathcal{C}^{\mathrm{fp}}$ is a monoidal subcategory. Induce a monoidal structure on $\mathcal{C}^{\mathrm{fp}}\hbox{-}\mathrm{Mod}$ via Day convolution product and note that this restricts to a monoidal structure on $\mathcal{C}^{\mathrm{fp}}\hbox{-}\mathrm{mod}$. We denote all tensor products by $\otimes$. Note that the monoidal structures on $\mathcal{C}$ and $\mathcal{C}^{\mathrm{fp}}\hbox{-}\mathrm{Mod}$ are assumed to be closed, and therefore in both cases the tensor product functor, $\otimes$, is right exact in each variable. Furthermore, as $\mathcal{C}$ is an additive finitely accessible category with products, $\mathcal{C}^{\mathrm{fp}}\hbox{-}\mathrm{Mod}$ is locally coherent (\cite{P11}, Theorem 6.1) and therefore $\mathcal{C}^{\mathrm{fp}}\hbox{-}\mathrm{mod}$ is an abelian subcategory of $\mathcal{C}^{\mathrm{fp}}\hbox{-}\mathrm{Mod}$ (e.g. see \cite{BRB}, Theorem E.1.47). Therefore, every exact sequence in $\mathcal{C}^{\mathrm{fp}}\hbox{-}\mathrm{mod}$ is exact in $\mathcal{C}^{\mathrm{fp}}\hbox{-}\mathrm{Mod}$ and consequently the restriction of Day convolution product to $\mathcal{C}^{\mathrm{fp}}\hbox{-}\mathrm{mod}$ is also right exact in each variable.    
\end{assumption}

Before we prove Theorem \ref{thm}, we prove some useful Lemmas. The first uses the tensor-hom adjunction of a closed monoidal category to find a natural isomorphism between two functors. 

\begin{lemma} \label{lemma rep}
	Let $\mathcal{C}$ be as in Assumption \ref{Rmk C} and induce a monoidal structure on $\mathcal{C}^{\mathrm{fp}}\hbox{-}\mathrm{Mod}$ via Day convolution product. Then for all $F \in \mathcal{C}^{\mathrm{fp}}\hbox{-}\mathrm{Mod}$ and $X \in \mathcal{C}^{\mathrm{fp}}$, $(X,-) \otimes F$ is naturally isomorphic to $\overrightarrow{F} \circ \mathrm{hom}(X,-)|_{\mathcal{C}^{\mathrm{fp}}}$ where $\mathrm{hom}(X,-):\mathcal{C} \to \mathcal{C}$ denotes the internal hom-functor and $\mathrm{hom}(X,-)|_{\mathcal{C}^{\mathrm{fp}}}:\mathcal{C}^{\mathrm{fp}} \to \mathcal{C}$ is the restriction to finitely presented objects.  
\end{lemma}

\begin{proof}
	First suppose $F$ is finitely presentable with presentation $(B,-) \xrightarrow{(f,-)} (A,-) \xrightarrow{\pi_f} F \to 0$, with $f:A \to B$ in $\mathcal{C}^{\mathrm{fp}}$ and suppose $X \in \mathcal{C}^{\mathrm{fp}}$. Then $(X,-) \otimes F$ has presentation $(X \otimes B,-) \xrightarrow{(X \otimes f,-)} (X \otimes A,-) \xrightarrow{\pi_{X \otimes f}} (X,-) \otimes F \to 0$, where $X \otimes f:X \otimes A \to X \otimes B$ is in $\mathcal{C}^{\mathrm{fp}}$. For any $Z \in \mathcal{C}^{\mathrm{fp}}$ we have the following diagram in $\mathbf{Ab}$.
	
	\noindent \begin{tikzpicture}
	\matrix (m) [matrix of math nodes,row sep=5em,column sep=4em,minimum width=2em]
	{
		(X \otimes B,Z) & (X \otimes A,Z) & ((X,-) \otimes F)(Z) & 0 \\
		(B, \mathrm{hom}(X,Z)) & (A, \mathrm{hom}(X,Z)) & (\overrightarrow{F} \circ \mathrm{hom}(X,-)|_{\mathcal{C}^{\mathrm{fp}}})(Z) & 0 \\};
	\path[-stealth]
	(m-1-1) edge node [above] {$(X \otimes f,-)_Z$} (m-1-2)
	(m-1-2) edge node [above] {$(\pi_{X \otimes f})_Z$} (m-1-3)
	(m-1-3) edge (m-1-4)
	(m-2-1) edge node [below] {$(f, \mathrm{hom}(X,-))_Z$} (m-2-2)
	(m-2-2) edge node [below] {$(\pi_{f})_{\mathrm{hom}(X,Z)}$} (m-2-3)
	(m-2-3) edge (m-2-4)
	(m-1-1) edge node [right] {$\alpha_B$} (m-2-1)
	(m-1-2) edge node [right] {$\alpha_A$}(m-2-2)
	(m-1-3) edge [dashed] node [right] {$\eta_Z$} (m-2-3);
	\end{tikzpicture}
	
	Since the isomorphisms $\alpha_{B}$ and $\alpha_{A}$ are natural in $A$ and $B$ respectively, the $\eta_Z$ form the components of a natural isomorphism $\eta:(X,-) \otimes F \to \overrightarrow{F} \circ \mathrm{hom}(X,-)|_{\mathcal{C}^{\mathrm{fp}}}$.  
		
	If $F:\mathcal{C}^{\mathrm{fp}} \to \mathbf{Ab}$ is any additive functor then $F=\underrightarrow{\mathrm{lim}}_{i \in I} F_i$ for some finitely presented functors $F_i$. For each $i \in I$, we have $\eta_i:((X,-) \otimes F_i) \to \overrightarrow{F_i} \circ \mathrm{hom}(X,-)|_{\mathcal{C}^{\mathrm{fp}}}$, defined as above. Furthermore, for any natural transformation $\lambda:F_{i} \to F_{j}$ the following diagram commutes, where $\overrightarrow{\lambda}:\overrightarrow{F_i} \to \overrightarrow{F_j}$ denotes the natural transformation with components given by the unique map between direct limits induced by $\lambda$.
	
	\noindent \begin{tikzpicture}
	\matrix (m) [matrix of math nodes,row sep=5em,column sep=4em,minimum width=2em]
	{
		(X,-) \otimes F_i & \overrightarrow{F_i} \circ \mathrm{hom}(X,-)|_{\mathcal{C}^{\mathrm{fp}}} \\
		(X, -) \otimes F_j & \overrightarrow{F_j} \circ \mathrm{hom}(X,-)|_{\mathcal{C}^{\mathrm{fp}}} \\};
	\path[-stealth]
	(m-1-1) edge node [above] {$\eta_i$} (m-1-2)
	(m-2-1) edge node [below] {$\eta_j$} (m-2-2)
	(m-1-1) edge node [left] {$(X,-) \otimes \lambda$} (m-2-1)
	(m-1-2) edge node [right] {$\overrightarrow{\lambda}_{\mathrm{hom}(X,-)|_{\mathcal{C}^{\mathrm{fp}}}} $}(m-2-2);
	\end{tikzpicture}
	
	Therefore, by the universal property of direct limits, the $\eta_i$ for $i \in I$ induce a unique natural isomorphism \[\underrightarrow{\mathrm{lim}}_{i \in I} ((X,-) \otimes F_i) \to \underrightarrow{\mathrm{lim}}_{i \in I} (\overrightarrow{F_i} \circ \mathrm{hom}(X,-)|_{\mathcal{C}^{\mathrm{fp}}})=\overrightarrow{F} \circ \mathrm{hom}(X,-)|_{\mathcal{C}^{\mathrm{fp}}}.\] Since $(X,-) \otimes -$ commutes with direct limits, \[ \underrightarrow{\mathrm{lim}}_{i \in I} ((X,-) \otimes F_i) \cong  (X,-) \otimes \underrightarrow{\mathrm{lim}}_{i \in I} F_i=(X,-) \otimes F.\] Therefore, we have a natural isomorphism $\eta: (X,-) \otimes F \to \overrightarrow{F} \circ \mathrm{hom}(X,-)|_{\mathcal{C}^{\mathrm{fp}}}$ as required.        
\end{proof}

\begin{lemma} \label{lem hom(a,-) work around}
	Let $\mathcal{C}$ be as in Assumption \ref{Rmk C}. For every $F \in \mathcal{C}^{\mathrm{fp}}\hbox{-}\mathrm{mod}$ and $C \in \mathcal{C}^{\mathrm{fp}}$, \[\overrightarrow{F} \circ \mathrm{hom}(C,-):\mathcal{C} \to \mathbf{Ab}\] commutes with direct limits. 	
\end{lemma}

\begin{proof}
	Suppose $F=F_f$ has presentation $(B,-) \xrightarrow{(f,-)} (A,-) \xrightarrow{\pi_f} F \to 0$ and express $Z \in \mathcal{C}$ as a direct limit of the directed diagram $\big( (Z_i)_{i \in I}, (\lambda_{ij}:Z_i \to Z_j)_{i \leq j}) \big)$, where each $Z_i$ is finitely presented and the limits maps are denoted by $\lambda_i:Z_i \to Z$. By (\cite{BRB}, Corollary 10.2.42), we have an exact sequence \[(B,\mathrm{hom}(C,Z)) \xrightarrow{(f,-)_{\mathrm{hom}(C,Z)}} (A,\mathrm{hom}(C,Z)) \xrightarrow{{\pi}_{\mathrm{hom}(C,Z)}} \overrightarrow{F}(\mathrm{hom}(C,Z)) \to 0.\]
	
	 By the adjunction isomorphisms and the fact that $A \otimes C$ is finitely presented we have isomorphisms 
	 \begin{align*}
	 (A,\mathrm{hom}(C,Z)) \cong (C \otimes A, Z) &=  (C \otimes A, \underrightarrow{\mathrm{lim}}_{i \in I} Z_i) 
	 \\ & \cong \underrightarrow{\mathrm{lim}}_{i \in I} (C \otimes A, Z_i)  \cong \underrightarrow{\mathrm{lim}}_{i \in I} (A, \mathrm{hom}(C,Z_i))
	 \end{align*}  and similarly with $A$ replaced by $B$. Furthermore these isomorphisms fit into the following commutative square.
	 
		\noindent \begin{tikzpicture}
	\matrix (m) [matrix of math nodes,row sep=5em,column sep=8em,minimum width=2em]
	{
		(B,\mathrm{hom}(C,Z)) & (A,\mathrm{hom}(C,Z)) \\
		 \underrightarrow{\mathrm{lim}}_{i \in I}(B,  \mathrm{hom}(C,Z_i)) & \underrightarrow{\mathrm{lim}}_{i \in I} (A,  \mathrm{hom}(C,Z_i)) \\};
	\path[-stealth]
	(m-1-1) edge node [above] {$(f,\mathrm{hom}(C,Z))$} (m-1-2)
	(m-2-1) edge node [below] {$\underrightarrow{\mathrm{lim}}_{i \in I}(f,   \mathrm{hom}(C,Z_i))$} (m-2-2)
	(m-1-1) edge node [left] {$\cong$} (m-2-1)
	(m-1-2) edge node [right] {$\cong$}(m-2-2);
	\end{tikzpicture}  
	
	Therefore the cokernels of the two horizontal maps, $\overrightarrow{F}(\mathrm{hom}(C,Z))=(\overrightarrow{F} \circ \mathrm{hom}(C,-))(\underrightarrow{\mathrm{lim}}_{i \in I} Z_i)$ and $\underrightarrow{\mathrm{lim}}_{i \in I}(\overrightarrow{F}(\mathrm{hom}(C,Z_i))$, are also isomorphic. 
\end{proof}

\begin{cor}
	Let $\mathcal{C}$ be as in Assumption \ref{Rmk C}. For all $C \in \mathcal{C}^{\mathrm{fp}}$ and $Z \in \mathcal{C}$, if $Z=\underrightarrow{\mathrm{lim}}_{i \in I}Z_i$ then $\mathrm{hom}(C,Z)$ and $\underrightarrow{\mathrm{lim}}_{i \in I}\mathrm{hom}(C,Z_i)$ belong to the same definable subcategories.
\end{cor}

\begin{remark}
	Note that if $\mathcal{C}^{\mathrm{fp}}$ is a rigid monoidal subcategory of $\mathcal{C}$, then for any $C \in \mathcal{C}^{\mathrm{fp}}$, $\mathrm{hom}(C,-)\cong C^{\vee} \otimes -$ commutes with direct limits. Therefore for every $Z \in \mathcal{C}$ satisfying $Z=\underrightarrow{\mathrm{lim}}_{i \in I}Z_i$, $\mathrm{hom}(C,Z)$ and $\underrightarrow{\mathrm{lim}}_{i \in I}\mathrm{hom}(C,Z_i)$ are isomorphic.
\end{remark}

Next, we simplify the criteria for a Serre subcategory of $\mathcal{C}^{\mathrm{fp}}\hbox{-}\mathrm{mod}$ to be a Serre tensor-ideal. 

\begin{lemma} \label{definable tensor ideal respresentables}
	Let $\mathcal{C}$ be as in Assumption \ref{Rmk C} and suppose $\mathsf{S}$ is a Serre subcategory of $\mathcal{C}^{\mathrm{fp}}\hbox{-}\mathrm{mod}$.
	
	Then $\mathsf{S}$ is a Serre tensor-ideal of $\mathcal{C}^{\mathrm{fp}}\hbox{-}\mathrm{mod}$ if and only if for all $C \in \mathcal{C}^{\mathrm{fp}}$ and all $F \in \mathsf{S}$,
	$$(C,-) \otimes F \in \mathsf{S}.$$
\end{lemma}

\begin{proof}
	$(\implies)$ Holds by the definition of tensor-ideal. 
	
	\noindent $(\impliedby)$ Suppose that for all $C \in \mathcal{C}^{\mathrm{fp}}$ and all $F \in \mathsf{S}$,
	$(C,-) \otimes F \in \mathsf{S}$. Let $F_f \in \mathsf{S}$ and $F_g \in \mathcal{C}^{\mathrm{fp}}\hbox{-}\mathrm{mod}$ where $F_g$ has projective resolution
	$$(V,-) \xrightarrow{(g,-)} (U,-) \to F_g \to 0,$$ \noindent for $g:U \to V$ a morphism in $\mathcal{C}^{\mathrm{fp}}$.
	
	By right exactness of the induced tensor product, we have the exact sequence
	$$(V,-) \otimes F_f \to (U,-) \otimes F_f \to F_g \otimes F_f \to 0.$$
	
	By assumption, $(U,-) \otimes F_f$ is an object of $\mathsf{S}$. Therefore $F_g \otimes F_f \in \mathsf{S}$ as $\mathsf{S}$ is Serre. Hence $\mathsf{S}$ is tensor-closed as required.
\end{proof}

Now we are ready to prove Theorem \ref{thm}.

\begin{theorem} \label{thm}
		Let $\mathcal{C}$ be an additive finitely accessible category with products. Suppose further that $\mathcal{C}$ has an additive closed symmetric monoidal structure such that $\mathcal{C}^{\mathrm{fp}}$ is a monoidal subcategory. 
	
	Let $\mathcal{D}$ be a definable subcategory of $\mathcal{C}$ and $\mathsf{S} \subseteq \mathcal{C}^{\mathrm{fp}}\hbox{-}\mathrm{mod}$ be the corresponding Serre subcategory as in Theorem \ref{loc fp def Serre}. Then $\mathsf{S}$ is a tensor-ideal of $\mathcal{C}^{\mathrm{fp}}\hbox{-}\mathrm{mod}$ if and only if $\mathcal{D}$ is fp-hom-closed. 
\end{theorem}

\begin{proof}
	Recall that the functors in $\mathsf{S}$ are exactly those whose unique extension along direct limits annihilates $\mathcal{D}$. Therefore $\mathcal{D}$ is fp-hom-closed if and only if for every $A \in \mathcal{C}^{\mathrm{fp}}$, $X \in \mathcal{D}$ and every $F \in \mathsf{S}$, $\overrightarrow{F}(\mathrm{hom}(A,X))=0$. By Lemma \ref{lem hom(a,-) work around}, $\overrightarrow{F} \circ \mathrm{hom}(A,-)$ commutes with direct limits and therefore $\overrightarrow{F} \circ \mathrm{hom}(A,-)=\overrightarrow{\overrightarrow{F} \circ \mathrm{hom}(A,-)|_{\mathcal{C}^{\mathrm{fp}}}}$. Furthermore, by Lemma \ref{lemma rep}, we have $(A,-) \otimes F \cong \overrightarrow{F} \circ \mathrm{hom}(A,-)|_{\mathcal{C}^{\mathrm{fp}}}:\mathcal{C}^{\mathrm{fp}} \to \mathbf{Ab}$, so $\overrightarrow{F} \circ \mathrm{hom}(A,-) \cong \overrightarrow{(A,-) \otimes F}$.
	
	Therefore, $\mathcal{D}$ is fp-hom-closed if and only if for every $A \in \mathcal{C}^{\mathrm{fp}}$, $X \in \mathcal{D}$ and $F \in \mathsf{S}$, $\overrightarrow{(A,-) \otimes F}(X)=0$, equivalently $(A,-) \otimes F \in \mathsf{S}$.
	
	Finally note that $\mathsf{S}$ is a Serre tensor-ideal if and only if it is closed under tensoring with representable functors (see Lemma \ref{definable tensor ideal respresentables}). 
\end{proof}	

By Theorem \ref{thm}, if $(\mathcal{D},\mathcal{C},\otimes)$ is an object of $\mathbb{DEF}^{\otimes}$, then the corresponding Serre subcategory of $\mathcal{C}^{\mathrm{fp}}\hbox{-}\mathrm{mod}$ is a Serre tensor-ideal. Next we use this to define a monoidal structure on $\mathrm{fun}(\mathcal{D})$. We prove the following lemma first.

\begin{lemma} \label{(x,-) otimes - exact}
Let $\mathcal{C}$ be as in Assumption \ref{Rmk C}. Then $(C,-) \otimes -:\mathcal{C}^{\mathrm{fp}}\hbox{-}\mathrm{Mod} \to \mathcal{C}^{\mathrm{fp}}\hbox{-}\mathrm{Mod}$ is exact for all $C \in \mathcal{C}^{\mathrm{fp}}$.

If we assume further that $C \in \mathcal{C}^{\mathrm{fp}}$ is rigid then so is $(C,-)$ with dual given by $(C^{\vee},-)$. 
\end{lemma}

\begin{proof}
We already know that $(C,-) \otimes -$ is a left adjoint (\cite{D70}, Theorem 3.3 and Theorem 3.6) and therefore right exact. We show that $(C,-) \otimes -$ is also a right adjoint and therefore is an exact functor.

We first define the left adjoint $\mathbf{L}_C:\mathcal{C}^{\mathrm{fp}}\hbox{-}\mathrm{Mod} \to \mathcal{C}^{\mathrm{fp}}\hbox{-}\mathrm{Mod}$ on finitely presented functors. Given $F_f \in \mathcal{C}^{\mathrm{fp}}\hbox{-}\mathrm{mod}$ with presentation $(B,-) \xrightarrow{(f,-)} (A,-) \to F_f \to 0$, denote by $\mathbf{L}_C(F_f)$ the functor with presentation \[ (\mathrm{hom}(C,B),-) \xrightarrow{(\mathrm{hom}(C,f),-)} (\mathrm{hom}(C,A),-) \to \mathbf{L}_C(F_f) \to 0.\] It can be checked that this definition does not depend on the choice of $f$.   

Now, given another finitely presented functor $F_g$ with presentation $(V,-) \xrightarrow{(g,-)} (U,-) \to F_g \to 0$ and a morphism $\alpha:F_f \to F_g$, chose any $\alpha_1:U \to A$ and $\alpha_2:V \to B$ such that the following diagram commutes.

\begin{tikzpicture}
\matrix (m) [matrix of math nodes,row sep=4em,column sep=5em,minimum width=3em]
{ 0 & 0 \\
	F_f & F_g \\
	(A,-) & (U,-) \\
	(B,-) & (V,-) \\};
\path[-stealth]
(m-2-1) edge (m-1-1)
(m-2-2) edge (m-1-2)
(m-2-1) edge node [above] {$\alpha$} (m-2-2)
(m-3-1) edge node [above] {$(\alpha_1,-)$} (m-3-2)
(m-4-1) edge node [above] {$(\alpha_2,-)$} (m-4-2)
(m-3-1) edge node [left] {$\pi_{f}$} (m-2-1)
(m-3-2) edge node [right] {$\pi_{g}$} (m-2-2)
(m-4-1) edge node [left] {$(f,-)$} (m-3-1)
(m-4-2) edge node [right] {$(g,-)$} (m-3-2);
\end{tikzpicture}

Then define the morphism $\mathbf{L}_C(\alpha):\mathbf{L}_C(F_f) \to \mathbf{L}_C(F_g)$ by the unique map making the following diagram commute.

\begin{tikzpicture}
\matrix (m) [matrix of math nodes,row sep=4em,column sep=6em,minimum width=3em]
{ 0 & 0 \\
	\mathbf{L}_C(F_f) & \mathbf{L}_C(F_g) \\
	(\mathrm{hom}(C,A),-) & (\mathrm{hom}(C,U),-) \\
	(\mathrm{hom}(C,B),-) & (\mathrm{hom}(C,V),-) \\};
\path[-stealth]
(m-2-1) edge (m-1-1)
(m-2-2) edge (m-1-2)
(m-2-1) edge node [above] {$\mathbf{L}_C(\alpha)$} (m-2-2)
(m-3-1) edge node [above] {$(\mathrm{hom}(C,\alpha_1),-)$} (m-3-2)
(m-4-1) edge node [above] {$(\mathrm{hom}(C,\alpha_2),-)$} (m-4-2)
(m-3-1) edge node [left] {$\pi_{\mathrm{hom}(C,f)}$} (m-2-1)
(m-3-2) edge node [right] {$\pi_{\mathrm{hom}(C,g)}$} (m-2-2)
(m-4-1) edge node [left] {$(\mathrm{hom}(C,f),-)$} (m-3-1)
(m-4-2) edge node [right] {$(\mathrm{hom}(C,g),-)$} (m-3-2);
\end{tikzpicture}

Note that this does not depend of the choice of $\alpha_1$ and $\alpha_2$. To define $\mathbf{L}_C$ (up to isomorphism) on any $F \in \mathcal{C}^{\mathrm{fp}}\hbox{-}\mathrm{Mod}$, we assert that $\mathbf{L}_C$ commutes with direct limits. 

 It is easy to check that $\mathbf{L}_C:\mathcal{C}^{\mathrm{fp}}\hbox{-}\mathrm{Mod} \to \mathcal{C}^{\mathrm{fp}}\hbox{-}\mathrm{Mod}$ defines (up to isomorphism) a functor, indeed this follows from the fact that $\mathrm{hom}(C,-):\mathcal{C} \to \mathcal{C}$ is a functor. We claim that this functor is left adjoint to $(C,-) \otimes -:\mathcal{C}^{\mathrm{fp}}\hbox{-}\mathrm{Mod} \to \mathcal{C}^{\mathrm{fp}}\hbox{-}\mathrm{Mod}$.

As $(C,-) \otimes -$ and $\mathbf{L}_C$ commute with direct limits, it is enough to define the unit and counit of the adjunction on finitely presented functors. Indeed, any functor $F \in {\mathcal{C}^{\mathrm{fp}}\hbox{-}\mathrm{Mod}}$ can be expressed as a direct limit of finitely presented functors, say $F=\underrightarrow{\mathrm{lim}}_{i \in I} F_i$ where each $F_i \in \mathcal{C}^{\mathrm{fp}}\hbox{-}\mathrm{mod}$. By the universal property of direct limits, the value of the unit, $\boldsymbol{\eta}_F:F \to  ((C,-) \otimes \mathbf{L}_C(F)$, and the counit, $\boldsymbol{\varepsilon}_F:\mathbf{L}_C((C,-) \otimes F) \to F$, at $F$, is uniquely determined by the respective components at the $F_i$.   

The unit, $\boldsymbol{\eta}:\mathrm{Id}_{\mathcal{C}^{\mathrm{fp}}\hbox{-}\mathrm{Mod}} \to ((C,-) \otimes -) \circ \mathbf{L}_C$, is defined on finitely presented functors as follows. For $F_f \in \mathcal{C}^{\mathrm{fp}}\hbox{-}\mathrm{mod}$ we define $\boldsymbol{\eta}_{F_f}:F_f \to (C,-) \otimes \mathbf{L}_C(F_f)$ as the unique map such that the following diagram commutes, where $\varepsilon_C:(C \otimes -) \circ \mathrm{hom}(C,-) \to \mathrm{Id}_{\mathcal{C}}$ is the counit of the adjunction between $C \otimes -$ and $\mathrm{hom}(C,-)$. 

\begin{tikzpicture}
\matrix (m) [matrix of math nodes,row sep=4em,column sep=5em,minimum width=3em]
{ 0 & 0 \\
	F_f & (C,-) \otimes \mathbf{L}_C(F_f) \\
	(A,-) & (C \otimes \mathrm{hom}(C,A),-) \\
	(B,-) & (C \otimes \mathrm{hom}(C,B),-) \\};
\path[-stealth]
(m-2-1) edge (m-1-1)
(m-2-2) edge (m-1-2)
(m-2-1) edge node [above] {$\boldsymbol{\eta}_{F_f}$} (m-2-2)
(m-3-1) edge node [above] {$((\varepsilon_C)_A,-)$} (m-3-2)
(m-4-1) edge node [above] {$((\varepsilon_C)_B,-)$} (m-4-2)
(m-3-1) edge node [left] {$\pi_{f}$} (m-2-1)
(m-3-2) edge node [right] {$(C,-) \otimes \pi_{\mathrm{hom}(C,f)}$} (m-2-2)
(m-4-1) edge node [left] {$(f,-)$} (m-3-1)
(m-4-2) edge node [right] {$(C \otimes \mathrm{hom}(C,f),-)$} (m-3-2);
\end{tikzpicture} 

Similarly we define the counit of the adjunction, $\boldsymbol{\varepsilon}:\mathbf{L}_C \circ (C,-) \otimes - \to \mathrm{Id}_{\mathcal{C}^{\mathrm{fp}}\hbox{-}\mathrm{Mod}}$, as follows. For $F_f \in \mathcal{C}^{\mathrm{fp}}\hbox{-}\mathrm{mod}$ we define $\boldsymbol{\varepsilon}_{F_f}:\mathbf{L}_C((C,-)) \otimes F_f) \to F_f$ as the unique map such that the following diagram commutes, where $\eta_C:\mathrm{Id}_{\mathcal{C}} \to \mathrm{hom}(C,-) \circ (C \otimes -)$ is the unit of the adjunction between $C \otimes -$ and $\mathrm{hom}(C,-)$.

\begin{tikzpicture}
\matrix (m) [matrix of math nodes,row sep=4em,column sep=5em,minimum width=3em]
{ 0 & 0 \\
	\mathbf{L}_C((C,-) \otimes F_f) & F_f \\
	(A,-) & (C \otimes \mathrm{hom}(C,A),-) \\
	(B,-) & (C \otimes \mathrm{hom}(C,B),-) \\};
\path[-stealth]
(m-2-1) edge (m-1-1)
(m-2-2) edge (m-1-2)
(m-2-1) edge node [above] {$\boldsymbol{\varepsilon}_{F_f}$} (m-2-2)
(m-3-1) edge node [above] {$((\eta_C)_A,-)$} (m-3-2)
(m-4-1) edge node [above] {$((\eta_C)_B,-)$} (m-4-2)
(m-3-1) edge node [left] {$\pi_{\mathrm{hom}(C,C \otimes f)}$} (m-2-1)
(m-3-2) edge node [right] {$\pi_{f}$} (m-2-2)
(m-4-1) edge node [left] {$( \mathrm{hom}(C,C \otimes f),-)$} (m-3-1)
(m-4-2) edge node [right] {$(f,-)$} (m-3-2);
\end{tikzpicture}  

It can be seen that $\boldsymbol{\eta}_{F_f}$ and $\boldsymbol{\varepsilon}_{F_f}$ don't depend of the choice of presentation. It remains to check that the triangle identities hold. Again, it is enough to check when evaluating at finitely presented functors, but these follow easily from the triangle identities on the adjunction between $C \otimes -$ and $\mathrm{hom}(C,-)$. 

If, in addition, we assume that $C \in \mathcal{C}^{\mathrm{fp}}$ is rigid then $\mathrm{hom}(C,-)\cong C^{\vee} \otimes -$ and therefore $\mathbf{L}_C \cong (C^{\vee},-) \otimes -$. Thus, $(C^{\vee},-) \otimes -$ is left adjoint to $(C,-) \otimes -$ and $(C,-) \otimes - \cong ((C^{\vee})^{\vee},-) \otimes -$ is left adjoint to $(C^{\vee},-) \otimes -$. Therefore, $(C,-) \otimes -$ is rigid with dual given by $(C^{\vee},-) \otimes -$ as required.
\end{proof}

Next we define an additive symmetric monoidal structure on $\mathrm{fun}(\mathcal{D})$.

\begin{definition} \label{remark strict}
	Suppose $(\mathcal{D},\mathcal{C}, \otimes) \in \mathbb{DEF}^{\otimes}$ and let $\mathsf{S} \subseteq \mathcal{C}^{\mathrm{fp}}\hbox{-}\mathrm{mod}$ be the Serre subcategory corresponding to $\mathcal{D}$. By Theorem \ref{thm}, $\mathsf{S}$ is a Serre tensor-ideal of $\mathcal{C}^{\mathrm{fp}}\hbox{-}\mathrm{mod}$.
	First we define an additive symmetric monoidal structure on $\mathcal{C}^{\mathrm{fp}}\hbox{-}\mathrm{mod}/\mathsf{S}$.
	
By \cite{D73}, if the multiplicative system $\Sigma_{\mathsf{S}}$ of all the morphisms $\alpha$ in $\mathcal{C}^{\mathrm{fp}}\hbox{-}\mathrm{mod}$ such that $\mathrm{ker}(\alpha)$, $\mathrm{coker}(\alpha) \in \mathsf{S}$ is closed under tensoring with objects of $\mathcal{C}^{\mathrm{fp}}\hbox{-}\mathrm{mod}$, then $\mathcal{C}^{\mathrm{fp}}\hbox{-}\mathrm{mod}/\mathsf{S}$ has a monoidal structure such that the localisation functor $q:\mathcal{C}^{\mathrm{fp}}\hbox{-}\mathrm{mod} \to \mathrm{fun}(\mathcal{D})=\mathcal{C}^{\mathrm{fp}}\hbox{-}\mathrm{mod}/\mathsf{S}$ is universal among \textit{monoidal} functors which map the morphisms in $\Sigma_{\mathsf{S}}$ to isomorphisms. 

 By Lemma \ref{(x,-) otimes - exact}, for any $C \in \mathcal{C}^{\mathrm{fp}}$, $(C,-) \otimes -:\mathcal{C}^{\mathrm{fp}}\hbox{-}\mathrm{Mod} \to \mathcal{C}^{\mathrm{fp}}\hbox{-}\mathrm{Mod}$ is exact. Since, $\mathcal{C}^{\mathrm{fp}}\hbox{-}\mathrm{mod}$ is an abelian subcategory of $\mathcal{C}^{\mathrm{fp}}\hbox{-}\mathrm{Mod}$, $(C,-) \otimes -:\mathcal{C}^{\mathrm{fp}}\hbox{-}\mathrm{mod} \to \mathcal{C}^{\mathrm{fp}}\hbox{-}\mathrm{mod}$ is also exact. Consequently, for every $\alpha:F \to G$ in $\mathcal{C}^{\mathrm{fp}}\hbox{-}\mathrm{mod}$, $\mathrm{ker}((C,-) \otimes \alpha)\cong (C,-) \otimes \mathrm{ker}(\alpha)$ and $\mathrm{coker}((C,-) \otimes \alpha)\cong (C,-) \otimes \mathrm{coker}(\alpha)$.  Therefore if $\mathsf{S}$ is a tensor-ideal and $\alpha \in \Sigma_{\mathsf{S}}$ then $\mathrm{ker}((C,-) \otimes \alpha),~\mathrm{coker}((C,-) \otimes \alpha) \in \mathsf{S}$ so $(C,-) \otimes \alpha \in \Sigma_{\mathsf{S}}$. 

 Now consider the morphism $\alpha \otimes Id_{F_g}$. We have the following commutative diagram with exact rows.

\begin{tikzpicture}
\matrix (m) [matrix of math nodes,row sep=5em,column sep=5em,minimum width=2em]
{
	F \otimes (V,-) & F \otimes (U,-) & F \otimes F_g & 0 \\
	G \otimes (V,-) & G \otimes (U,-) & G \otimes F_g & 0 \\};
\path[-stealth]
(m-1-1) edge node [above] {$Id_F \otimes (g,-)$} (m-1-2)
(m-1-2) edge (m-1-3)
(m-1-3) edge (m-1-4)
(m-2-1) edge node [below] {$Id_G \otimes (g,-)$} (m-2-2)
(m-2-2) edge (m-2-3)
(m-2-3) edge (m-2-4)
(m-1-1) edge node [right] {$\alpha \otimes Id_{(V,-)}$} (m-2-1)
(m-1-2) edge node [right] {$\alpha \otimes Id_{(U,-)}$}(m-2-2)
(m-1-3) edge [dashed] node [right] {$\alpha \otimes Id_{F_g}$} (m-2-3);
\end{tikzpicture}

Since for every $D \in \mathcal{D}$, $(\alpha \otimes Id_{(V,-)})_D$ and  $(\alpha \otimes Id_{(U,-)})_D$ are isomorphisms, so is $(\alpha \otimes Id_{F_g})_D$. Hence $\mathrm{ker}(\alpha \otimes Id_{F_g}),~ \mathrm{coker}(\alpha \otimes Id_{F_g}) \in \mathsf{S}$ and $\alpha \otimes Id_{F_g} \in \Sigma_{\mathsf{S}}$.

Applying (\cite{D73}, Corollary 1.4) we get an additive symmetric monoidal structure on $\mathcal{C}^{\mathrm{fp}}\hbox{-}\mathrm{mod}/\mathsf{S}$. We induce a monoidal structure on $\mathrm{fun}(\mathcal{D})=(\mathcal{D},\mathbf{Ab})^{\Pi \to}$ via the equivalence given in (\cite{P11}, Theorem 12.10). 
\end{definition}

Next we show that the monoidal structure on $\mathrm{fun}(\mathcal{D})$ is exact in each variable.

\begin{prop} \label{prop exactness only if}
	Let $\mathcal{C}$ be as in Assumption \ref{Rmk C}. Suppose $\mathcal{D}$ is an fp-hom-closed definable subcategory and induce a monoidal structure on $\mathrm{fun}(\mathcal{D})$ as in Remark \ref{remark strict}.
	
	If $\mathcal{D}$ satisfies the exactness criterion then the monoidal structure on $\mathrm{fun}(\mathcal{D})$ is exact in each variable. 
\end{prop}

\begin{proof}
	Suppose $\mathcal{D}$ satisfies the exactness criterion, i.e. for any $f:A \to B$ and $g:U \to V$ in $\mathcal{C}^{\mathrm{fp}}$ and for any $D \in \mathcal{D}$, if $h:A \otimes U \to D$ satisfies $(f \otimes U) | h$ and $(A \otimes g) | h$ then $(f \otimes g) | h$. Suppose further that $0 \to F \to G \to H \to 0$ is an exact sequence in $\mathrm{fun}(\mathcal{D})$. It is (isomorphic to) the image of an exact sequence $0 \to F_f \xrightarrow{\alpha} F_g \xrightarrow{\beta} F_l \to 0$ in  $\mathcal{C}^{\mathrm{fp}}\hbox{-}\mathrm{mod}$ (see \cite{BRB}, Lemma 11.1.6 and Corollary 11.1.42). If $K \in \mathrm{fun}(\mathcal{D})$ then $K$ is isomorphic to the image of $F_k$ for some $F_k \in \mathcal{C}^{\mathrm{fp}}\hbox{-}\mathrm{mod}$. Therefore, since the localisation functor is monoidal, showing $0 \to K \otimes F \xrightarrow{K \otimes \alpha} K \otimes G \xrightarrow{K \otimes \beta} K \otimes H \to 0$ is a short exact sequence in $\mathrm{fun}(\mathcal{D})$ is equivalent to showing that the image of the (not necessarily exact) sequence $0 \to F_k \otimes F_f \to F_k \otimes F_g \to F_k \otimes F_l \to 0$ under the localisation functor gives a short exact sequence. By (\cite{P11}, Theorem 12.10), this is equivalent to showing that $0 \to (\overrightarrow{F_k \otimes F_f})(D) \to (\overrightarrow{F_k \otimes F_g})(D) \to (\overrightarrow{F_k \otimes F_l})(D) \to 0$ is exact for all $D \in \mathcal{D}$. 
	
	Suppose $F_k$ has presentation $(T,-) \xrightarrow{(k,-)} (S,-) \to F_k \to 0$, where $k:S \to T$ is a morphism in $\mathcal{C}^{\mathrm{fp}}$, then we have the following commutative diagram in $\mathcal{C}^{\mathrm{fp}}\hbox{-}\mathrm{mod}$.
	
	\noindent	\begin{tikzpicture}
	\matrix (m) [matrix of math nodes,row sep=4em,column sep=4em,minimum width=3em]
	{  & 0 & 0 & 0 & \\
		& F_k \otimes F_f & F_k \otimes F_g & F_k \otimes F_l & 0 \\
		0 & (S,-) \otimes F_f & (S,-) \otimes F_g & (S,-) \otimes F_l & 0 \\
		0 & (T,-) \otimes F_f & (T,-) \otimes F_g & (T,-) \otimes F_l & 0 \\};
	\path[-stealth]
	(m-2-2) edge (m-1-2)
	(m-2-3) edge (m-1-3)
	(m-2-4) edge (m-1-4)
	(m-2-2) edge node [above] {$F_k \otimes \alpha$} (m-2-3)
	(m-2-3) edge node [above] {$F_k \otimes \beta$} (m-2-4)
	(m-2-4) edge (m-2-5)
	(m-3-2) edge node [left] {$\pi_k \otimes F_f$} (m-2-2)
	(m-3-3) edge (m-2-3)
	(m-3-4) edge (m-2-4)
	(m-3-1) edge (m-3-2)
	(m-3-4) edge (m-3-5)
	(m-3-2) edge node [above] {$(S,-) \otimes \alpha$} (m-3-3)
	(m-3-3) edge node [above] {$(S,-) \otimes \beta$} (m-3-4)
	(m-4-2) edge node [left] {$(k,-) \otimes F_f$} (m-3-2)
	(m-4-3) edge node [left] {$(k,-) \otimes F_g$} (m-3-3)
	(m-4-4) edge node [right] {$(k,-) \otimes F_l$} (m-3-4)
	(m-4-1) edge (m-4-2)
	(m-4-4) edge (m-4-5)
	(m-4-2) edge node [above] {$(T,-) \otimes \alpha$} (m-4-3)
	(m-4-3) edge node [above] {$(T,-) \otimes \beta$} (m-4-4);
	\end{tikzpicture} 
	
	Here the second row is exact since $F_k \otimes -$ is a left adjoint and therefore right exact. The third and fourth rows are exact by Lemma \ref{(x,-) otimes - exact}. We must show that $\overrightarrow{(F_k \otimes \alpha)}_D$ is a monomorphism (or has zero kernel) for all $D \in \mathcal{D}$. Fix $D \in \mathcal{D}$. To enhance readability, for a functor $F \in \mathcal{C}^{\mathrm{fp}}\hbox{-}\mathrm{mod}$ we will suppress the usual notation, $\overrightarrow{F}$, for the unique extension to a functor $\mathcal{C} \to \mathbf{Ab}$ which commutes with direct limits, and simply use $F$. For any $w \in (F_k \otimes F_f)(D)$ there exists some $w' \in ((S,-) \otimes F_f)(D)$ such that $(\pi_k \otimes F_f)_D(w')=w$. If $(F_k \otimes \alpha)_D(w)=0$ then there exists some $z \in ((T,-) \otimes F_g)(D)$ such that $((k,-) \otimes F_g)_D(z)=((S,-) \otimes \alpha)_D(w')$. 
	
	We will show that $w'=((k,-) \otimes F_f)_{D}(x)$ for some $x \in ((T,-) \otimes F_f)(D)$, meaning $w=(\pi_k \otimes F_f)(w')=0$.
	
	Suppose that $F_f$ has presentation $(B,-) \xrightarrow{(f,-)} (A,-) \xrightarrow{\pi_f} F_f \to 0$. Then there exists some morphism $w'':S \otimes A \to D$ such that $((S,-) \otimes \pi_f)_D(w'')=w'$. Similarly, if $g:U \to V$ in $\mathcal{C}^{\mathrm{fp}}$ then there exists some morphism $z':T \otimes U \to D$ such that $((T,-) \otimes \pi_g)_D(z')=z$. Therefore we have \[(((k,-) \otimes F_g) \circ ((T,-) \otimes \pi_g) )_D(z')=(((S,-) \otimes \alpha) \circ ((S,-) \otimes \pi_f))_D(w'').\] Consider the following diagram where $z'\in (T \otimes U, D)$ and $w'' \in (S \otimes A,D)$. 
	
	\noindent	\begin{tikzpicture}
	\matrix (m) [matrix of math nodes,row sep=4em,column sep=5em,minimum width=3em]
	{  (S \otimes B,D) & (S \otimes A,D) & ((S,-) \otimes F_f)(D) & 0 \\
		(S \otimes V, D)	& (S \otimes U,D)  & ((S,-) \otimes F_g)(D) & 0 \\
		(T \otimes V,D)  & (T \otimes U,D)  & ((T,-) \otimes F_g)(D) & 0 \\};
	\path[-stealth]
	(m-1-1) edge node [above] {$(S \otimes f,-)_D$} (m-1-2)
	(m-2-1) edge node [above] {$(S \otimes g,-)_D$} (m-2-2)
	(m-1-2) edge node [above] {$((S,-) \otimes \pi_{f})_D$} (m-1-3)
	(m-2-2) edge node [above] {$((S,-) \otimes \pi_{g})_D$} (m-2-3)
	(m-1-3) edge (m-1-4)
	(m-1-3) edge node [left] {$((S,-) \otimes \alpha)_D$} (m-2-3)
	(m-1-2) edge node [left] {$(S \otimes \alpha_1,-)_D$} (m-2-2)
	(m-3-3) edge node [left] {$((k,-) \otimes F_g)_D$} (m-2-3)
	(m-3-2) edge node [left] {$(k \otimes U,-)_D$} (m-2-2)
	(m-3-1) edge node [above] {$(T \otimes g,-)_D$} (m-3-2)
	(m-3-2) edge node [above] {$((T,-) \otimes\pi_{ g})_D$} (m-3-3)
	(m-3-3) edge (m-3-4)
	(m-2-3) edge (m-2-4);
	\end{tikzpicture}
	
	By projectivity of representables, there exists some $\alpha_1:U \to A$ such that $\alpha \circ \pi_f=\pi_g \circ (\alpha_1,-)$. Therefore, $((S,-) \otimes \alpha) \circ ((S,-) \otimes \pi_f)=((S-) \otimes \pi_g) \circ (S \otimes \alpha_1,-)$. This gives us \[(((S,-) \otimes \pi_g) \circ (k \otimes U,-) )_D(z')=(((S,-) \otimes \pi_g) \circ (S \otimes \alpha_1,-))_D(w'').\] That is $((S,-) \otimes \pi_g)_D(z' \circ (k \otimes U)-w'' \circ (S \otimes \alpha_1))=0$ and so there exists some $y:S \otimes V \to D$ such that $y \circ (S \otimes g)=z' \circ (k \otimes U)-w'' \circ (S \otimes \alpha_1)$ equivalently $h=y \circ (S \otimes g)+w'' \circ (S \otimes \alpha_1)=z' \circ (k \otimes U)$. 
	
	Consider the morphism $(g,\alpha_1):U \to V \oplus A$ such that $p_1 \circ (g,\alpha_1)=g$ and $p_2 \circ (g, \alpha_1)=\alpha_1$, where $p_1$ and $p_2$ denote the projection maps. Then $(S \otimes (g,\alpha_1)) | h$ and $(k \otimes U) | h$ so since the exactness criterion holds for $\mathcal{D}$, $(k \otimes (g,\alpha_1)) | h$ and there exists some $y':T \otimes (V \oplus A) \to D$ such that $y' \circ (k \otimes (g,\alpha_1))=h$. Set $y_1=y' \circ (T \otimes i_1)$ and $y_2=y' \circ (T \otimes i_2)$, where $i_1:V \to V \oplus A$ and $i_2:A \to V \oplus A$ are the inclusion maps. Then we have \[ y_1 \circ (k \otimes g) + y_2 \circ (k \otimes \alpha_1)=h=z' \circ (k \otimes U)=w'' \circ (S \otimes \alpha_1) + y \circ (S \otimes g).\]
	
	Applying $((S,-) \otimes \pi_g)_D$ we get, \[ ((S,-) \otimes \pi_g)_D(y_2 \circ (k \otimes \alpha_1))=((S,-) \otimes \pi_g)_D(h)=((S,-) \otimes \pi_g)_D(w'' \circ (S \otimes \alpha_1)).\]
	
	As shown on the commutative diagram above, we have \[((S,-) \otimes \pi_g) \circ ((S,-) \otimes (\alpha_1,-))=((S,-) \otimes \alpha) \circ ((S,-) \otimes \pi_f).\] Therefore, as \[ ((S,-) \otimes \pi_g)_D(y_2 \circ (k \otimes \alpha_1))=[((S,-) \otimes \pi_g) \circ ((S,-) \otimes (\alpha_1,-))]_D(y_2 \circ (k \otimes A))\] we have \[((S,-) \otimes \pi_g)_D(y_2 \circ (k \otimes \alpha_1))=[((S,-) \otimes \alpha) \circ ((S,-) \otimes \pi_f)]_D (y_2 \circ (k \otimes A)). \]
	
	Similarly, \[ ((S,-) \otimes \pi_g)_D(w'' \circ (S \otimes \alpha_1))=[((S,-) \otimes \pi_g) \circ ((S,-) \otimes (\alpha_1,-))]_{\mathcal{D}}(w''),\] so we have \[((S,-) \otimes \pi_g)_D(w'' \circ (S \otimes \alpha_1)) =[((S,-) \otimes \alpha) \circ ((S,-) \otimes \pi_f) ]_D(w'').\] 
	
	Therefore, \[[((S,-) \otimes \alpha) \circ ((S,-) \otimes \pi_f)]_D (y_2 \circ (k \otimes A))=[((S,-) \otimes \alpha) \circ ((S,-) \otimes \pi_f) ]_D(w'')\] and since $(S,-) \otimes \alpha$ is a monomorphism we have \[((S,-) \otimes \pi_f)_D (y_2 \circ (k \otimes A))=((S,-) \otimes \pi_f)_D(w'')=w'.\] 
	
	Finally note that \[((S,-) \otimes \pi_f)_D (y_2 \circ (k \otimes A))=[((S,-) \otimes \pi_f)\circ ((k,-) \otimes (A,-))]_D (y_2)\] and therefore, as \[((S,-) \otimes \pi_f)\circ ((k,-) \otimes (A,-))=((k,-) \otimes (F_f,-)) \circ ( ((T,-) \otimes \pi_f)\] we get that $w'=((k,-) \otimes (F_f,-))_{\mathcal{D}}(( ((T,-) \otimes \pi_f)_{\mathcal{D}}(y_2))$.  But then $w=(\pi_k \otimes F)_D(w')=(\pi_k \otimes F)_D \circ ((k,-) \otimes F_f)_D \circ ((T,-) \otimes \pi_f)_D(y_2)=0$. Therefore, $(F_k \otimes \alpha)_D$ has zero kernel and is a monomorphism for all $D \in \mathcal{D}$, as required. 
\end{proof}

In fact, the converse is also true.

\begin{prop} \label{Lem exactness if}
	Let $\mathcal{C}$ be as in Assumption \ref{Rmk C}. Suppose $\mathcal{D}$ is an fp-hom-closed definable subcategory and induce a monoidal structure on $\mathrm{fun}(\mathcal{D})$ as in Definition \ref{remark strict}.
	
	If the monoidal structure on $\mathrm{fun}(\mathcal{D})$ is exact in each variable then $\mathcal{D}$ satisfies the exactness criterion. 
\end{prop}

\begin{proof}
	Suppose that the induced monoidal structure on $\mathrm{fun}(\mathcal{D})$ is exact in each variable. Suppose $f:A \to B$ and $g:U \to V$ are morphisms in $\mathcal{C}^{\mathrm{fp}}$. By (\cite{P11}, Corollary 3.11), $\mathcal{C}^{\mathrm{fp}}$ has pseudocokernels. Let $g':V \to W$ be a pseudocokernel of $g$. 
	
	Then $(W,-) \xrightarrow{(g',-)} (V,-) \xrightarrow{(g,-)} (U,-)$ is exact in $\mathcal{C}^{\mathrm{fp}}\hbox{-}\mathrm{mod}$, that is $\mathrm{im}((g',-))=\mathrm{ker}((g,-))$. Therefore, its image $(W,-)_{\mathsf{S}} \xrightarrow{(g',-)_{\mathsf{S}}} (V,-)_{\mathsf{S}} \xrightarrow{(g,-)_{\mathsf{S}}} (U,-)_{\mathsf{S}}$ is exact in $\mathrm{fun}(\mathcal{D})$ and by assumption, $(F_f)_{\mathsf{S}} \otimes (W,-)_{\mathsf{S}} \xrightarrow{(F_f)_{\mathsf{S}} \otimes (g',-)_{\mathsf{S}}} (F_f)_{\mathsf{S}} \otimes (V,-)_{\mathsf{S}} \xrightarrow{(F_f)_{\mathsf{S}} \otimes (g,-)_\mathsf{S}} (F_f)_{\mathsf{S}} \otimes (U,-)_{\mathsf{S}}$ is also exact in $\mathrm{fun}(\mathcal{D})$. As the localisation functor is monoidal, this is equivalent to, $(F_f \otimes (W,-))(D) \xrightarrow{(F_f \otimes g',-))_D} (F_f \otimes (V,-))(D) \xrightarrow{(F_f \otimes (g,-))_D} (F_f \otimes (U,-))(D)$ being exact in $\mathbf{Ab}$ for all $D \in \mathcal{D}$, by (\cite{P11}, Theorem 12.10). Consider the diagram below.   
	
	\begin{tikzpicture}
	\matrix (m) [matrix of math nodes,row sep=4em,column sep=5em,minimum width=3em]
	{   0 & 0 & 0  \\
		F_f \otimes (W,-) & F_f \otimes (V,-) & F_f \otimes (U,-) \\
		(A \otimes W,-) & (A \otimes V,-) & (A \otimes U,-) \\
		(B \otimes W,-) & (B \otimes V,-) & (B \otimes U,-) \\};
	\path[-stealth]
	(m-2-1) edge (m-1-1)
	(m-2-2) edge (m-1-2)
	(m-2-3) edge (m-1-3)
	(m-2-1) edge node [above] {$F_f \otimes (g',-)$} (m-2-2)
	(m-2-2) edge node [above] {$F_f \otimes (g,-)$} (m-2-3)
	(m-3-1) edge node [left] {$\pi_f \otimes (W,-)$} (m-2-1)
	(m-3-2) edge node [left] {$\pi_f \otimes (V,-)$} (m-2-2)
	(m-3-3) edge  node [right] {$\pi_f \otimes (U,-)$} (m-2-3)
	(m-3-1) edge node [above] {$(A \otimes g',-) $} (m-3-2)
	(m-3-2) edge node [above] {$(A \otimes g,-)$} (m-3-3)
	(m-4-1) edge node [left] {$(f \otimes W,-)$} (m-3-1)
	(m-4-2) edge node [left] {$(f \otimes V,-)$} (m-3-2)
	(m-4-3) edge node [right] {$(f \otimes U,-)$} (m-3-3)
	(m-4-1) edge node [above] {$(B \otimes g',-)$} (m-4-2)
	(m-4-2) edge node [above] {$(B \otimes g,-)$} (m-4-3);
	\end{tikzpicture} 
	
	Given any $h:A \otimes U \to D$ such that $(f \otimes U) | h$ and $(A \otimes g) | h$ there exists $h_1:B \otimes U \to D$ such that $h=h_1 \circ (f \otimes U)$ and $h_2:A \otimes V \to D$ such that $h=h_2 \circ (A \otimes g)$. But this means that \begin{align*}((F_f \otimes (g,-)) \circ (\pi_f \otimes (V,-)))_D(h_2)&=(\pi_f \otimes (U,-))_D(h_2 \circ (A \otimes g))
	\\ &=(\pi_f \otimes (U,-))_D(h_1 \circ (f \otimes U))
	\\ &=0.\end{align*} So $(\pi_f \otimes (V,-))_D(h_2)$ is in the kernel of $(F_f \otimes (g,-))_D$ which is equal to the image of $(F_f \otimes (k,-))_D$.  Therefore there exists some $z \in (F_f \otimes (W,-))(D)$ such that $(F_f \otimes (k,-))_D(z)=(\pi_f \otimes (V,-))_D(h_2)$. But then, since $(\pi_f \otimes (W,-))_D$ is an epimorphism in $\mathbf{Ab}$, there exists some morphism $z':A \otimes W \to D$ such that $(\pi_f \otimes (W,-))_D(z')=z$. 
	
	Next, notice that $(\pi_f \otimes (V,-))_D(h_2-(z' \circ (A \otimes k)))=0$. So there exists some $y:B \otimes V \to D$ such that $y \circ (f \otimes V)=h_2-(z' \circ (A \otimes k))$. But then $y \circ (f \otimes g) =y \circ (f \otimes V) \circ (A \otimes g)=h_2 \circ (A \otimes g)=h$. That is $(f \otimes g) |h$, as required.
\end{proof}

\begin{prop}
	There exists a 2-functor $\theta:\mathbb{DEF}^{\otimes} \to \mathbb{ABEX}^{\otimes}$ which maps $(\mathcal{D},\mathcal{C},\otimes) \in \mathbb{DEF}^{\otimes}$ to $\mathrm{fun}(\mathcal{D})$ with the monoidal structure given in Definition \ref{remark strict}.
\end{prop}

\begin{proof}
	$\theta$ is well defined on objects by Theorem \ref{thm} and Proposition \ref{prop exactness only if}. On morphisms $\theta$ maps $I:\mathcal{D} \to \mathcal{D'}$ to $I_0:\mathrm{fun}(\mathcal{D'}) \to \mathrm{fun}(\mathcal{D})$ as in the original anti-equivalence in \cite{PR10}. $I_0$ is monoidal by definition of the morphisms $I$ in $\mathbb{DEF}^{\otimes}$. On natural transformations $\theta$ acts as in the original anti-equivalence and therefore $\theta$ satisfies the necessary axioms to be a 2-functor.
\end{proof}

\subsection{The 2-functor $\xi:\mathbb{ABEX}^{\otimes} \to \mathbb{DEF}^{\otimes}$}

Next we define a 2-functor $\xi:\mathbb{ABEX}^{\otimes} \to \mathbb{DEF}^{\otimes}$ which maps a skeletally small abelian category $\mathscr{A}$ with an additive symmetric monoidal structure to $(\mathrm{Ex}(\mathscr{A},\mathbf{Ab}), \mathscr{A}\hbox{-}\mathrm{Mod},\otimes)$, where $\otimes$ is induced by Day convolution product.

 First we show that $(\mathrm{Ex}(\mathscr{A},\mathbf{Ab}), \mathscr{A}\hbox{-}\mathrm{Mod},\otimes)$ is a well-defined object of $\mathbb{DEF}^{\otimes}$ (see Theorem \ref{TFAE SE tensor ideal}).

\begin{lemma} \label{lem monoidal equivalence}
	Let $\mathscr{A}$ be an additive symmetric monoidal, skeletally small abelian category. Suppose that for every exact functor $E:\mathscr{A} \to \mathbf{Ab}$, every $X \in \mathscr{A}$ and every short exact sequence $0 \to A \to B \to C \to 0$ in $\mathscr{A}$, \[ 0 \to E(X \otimes A) \to E(X \otimes B) \to E(X \otimes C) \to 0, \] \noindent is exact in $\mathbf{Ab}$. Suppose further that $\mathrm{Ex}(\mathscr{A},\mathbf{Ab})$ is fp-hom-closed and induce a monoidal structure on $\mathrm{fun}(\mathrm{Ex}(\mathscr{A},\mathbf{Ab}))$ as in Definition \ref{remark strict}. In addition, assume that the equivalence of categories $\mathscr{A} \simeq \mathrm{fun}(\mathrm{Ex}(\mathscr{A},\mathbf{Ab}))$ given by $A \mapsto \mathrm{ev}_A$, where $\mathrm{ev}_A:\mathrm{Ex}(\mathscr{A},\mathbf{Ab}) \to \mathbf{Ab}$ is given by `evaluation at $A$' (see \cite{PR10}), is monoidal.
	
	Then the monoidal structure on $\mathscr{A}$ is exact in each variable.
\end{lemma}

\begin{proof} 
Let $\mathscr{B}:=\mathrm{fun}(\mathrm{Ex}(\mathscr{A},\mathbf{Ab}))$. As there exists a monoidal equivalence between $\mathscr{A}$ and $\mathscr{B}$, the exactness assumption on $\mathscr{A}$ carries over to $\mathscr{B}$, that is, for every exact functor $E:\mathscr{B} \to \mathbf{Ab}$, every $X \in \mathscr{B}$ and every short exact sequence $0 \to A \to B \to C \to 0$ in $\mathscr{B}$, \[ 0 \to E(X \otimes A) \to E(X \otimes B) \to E(X \otimes C) \to 0, \] \noindent is exact in $\mathbf{Ab}$. \begin{comment} Indeed, suppose $0 \to F \to G \to 
	H \to 0$ is an exact sequence in $\mathrm{fun}(\mathrm{Ex}(\mathscr{A},\mathbf{Ab}))$, $K \in \mathrm{fun}(\mathrm{Ex}(\mathscr{A},\mathbf{Ab}))$ and $E:\mathrm{fun}(\mathrm{Ex}(\mathscr{A},\mathbf{Ab})) \to \mathbf{Ab}$ is an exact functor. Since $\vartheta \circ \varrho \cong \mathrm{Id}_{\mathrm{fun}(\mathrm{Ex}(\mathscr{A},\mathbf{Ab}))}$,  $0 \to E(K \otimes F) \to E(K \otimes G) \to E(K \otimes H) \to 0$ is exact if and only if \[0 \to (E \circ \vartheta \circ \varrho)(K \otimes F) \to (E \circ \vartheta  \circ \varrho)(K \otimes G) \to (E \circ \vartheta  \circ \varrho)(K \otimes H) \to 0\] is exact. But $E \circ \vartheta:\mathscr{A} \to \mathbf{Ab}$ is exact and $\varrho$ is monoidal and exact, so the last exact sequence is isomorphic to \[ 0 \to (E \circ \vartheta)(\varrho(K) \otimes \varrho(F)) \to (E \circ \vartheta)(\varrho(K) \otimes \varrho(G)) \to (E \circ \vartheta)(\varrho(K) \otimes \varrho(H)) \to 0,\] which is exact by our assumption on $\mathscr{A}$. 
	\end{comment}
	
	Set $\mathcal{D}:=\mathrm{Ex}(\mathscr{A},\mathbf{Ab})$ so $\mathscr{B}=\mathrm{fun}(\mathcal{D})$. For every $D \in \mathcal{D}$, $E=\mathrm{ev}_{D}:\mathscr{B}=\mathrm{fun}(\mathcal{D}) \to \mathbf{Ab}$ is exact, where $\mathrm{ev}_D$ is given by `evaluation at $D$'. Therefore, for every exact sequence $0 \to F \to G \to H \to 0$ in $\mathrm{fun}(\mathcal{D})$ and every $K \in \mathrm{fun}(\mathcal{D})$, $0 \to (K \otimes F)(D) \to (K \otimes G)(D) \to (K \otimes H)(D) \to 0$ is exact for all $D \in \mathcal{D}$. Since exactness in functor categories is pointwise, $0 \to K \otimes F \to K \otimes G \to K \otimes H \to 0$ is exact so the monoidal structure on $\mathscr{B}=\mathrm{fun}(\mathcal{D})$ is exact in each variable.
	
	As we are assuming the equivalence between $\mathscr{A}$ and $\mathscr{B}=\mathrm{fun}(\mathrm{Ex}(\mathscr{A},\mathbf{Ab}))$ is monoidal, the monoidal structure on $\mathscr{A}$ is also exact in each variable, as required.   
\end{proof}

\begin{theorem} \label{TFAE SE tensor ideal}
	Let $\mathscr{A}$ be an additive symmetric monoidal, skeletally small abelian category. The following are equivalent:
	\begin{enumerate}[label=(\roman*)]
		\item The definable subcategory $\mathrm{Ex}(\mathscr{A},\mathbf{Ab}) \subseteq \mathscr{A}\hbox{-}\mathrm{Mod}$ is fp-hom-closed (with respect to Day convolution product).
		
		\item The Serre subcategory $\mathsf{S}_{\mathrm{Ex}} \subseteq (\mathscr{A}\hbox{-}\mathrm{mod},\mathbf{Ab})^{\mathrm{fp}}$ of all functors $F$ such that $\overrightarrow{F}(E)=0$ for all $E \in \mathrm{Ex}(\mathscr{A}, \mathbf{Ab})$ is a tensor-ideal of $(\mathscr{A}\hbox{-}\mathrm{mod},\mathbf{Ab})^{\mathrm{fp}}$ (with respect to day convolution product). 
		
		\item The tensor product on $\mathscr{A}$ is exact in each variable.
	\end{enumerate}
\end{theorem}

\begin{proof}
	((i) $\iff$ (ii)) Follows directly from Theorem \ref{thm}.
	
	((iii) $\implies$ (i)) Suppose the monoidal structure on $\mathscr{A}$ is exact in each variable. We first show that $\mathrm{Ex}(\mathscr{A},\mathbf{Ab})$ is closed under $\mathrm{hom}(M,-)$ where $M \in \mathscr{A}\hbox{-}\mathrm{mod}$ is representable, say $M=(X,-)$. Indeed, in this case, for all $A \in \mathscr{A}$ and $E \in \mathrm{Ex}(\mathscr{A},\mathbf{Ab})$, \[\mathrm{hom}((X,-),E)(A) \cong ((A,-),\mathrm{hom}((X,-),E)) \cong ((A \otimes X,-),E)\cong E(A \otimes X),\] by the Yoneda lemma and adjunction isomorphisms. What's more, all these isomorphisms are natural in $A$. Therefore, \[ 0 \to \mathrm{hom}((X,-),E)(A) \to \mathrm{hom}((X,-),E)(B) \to \mathrm{hom}((X,-),E)(C) \to 0\] is exact if and only if \[0 \to E(A \otimes X) \to E(B \otimes X) \to E(C \otimes X) \to 0\] is exact. But the latter statement holds by our assumption on $\mathscr{A}$, as $E$ is an exact functor. Therefore $\mathrm{hom}((X,-),E)$ is exact as required.
	
	Now we generalise to $F_f \in \mathscr{A}\hbox{-}\mathrm{mod}$. We want to show that $\mathrm{hom}(F_f,E):\mathscr{A} \to \mathbf{Ab}$ is an exact functor. 
	
	First note that $(F_f,-)|_{\mathrm{Ex}(\mathscr{A},\mathbf{Ab})}$ commutes with direct products and direct limits and therefore is an object of $\mathrm{fun}(\mathrm{Ex}(\mathscr{A},\mathbf{Ab}))$. By (\cite{PR10}, Theorem 2.2), there exists an equivalence $\mathscr{A} \simeq \mathrm{fun}(\mathrm{Ex}(\mathscr{A},\mathbf{Ab}))$ given by $A \mapsto \mathrm{ev}_A$, where $\mathrm{ev}_A:\mathrm{Ex}(\mathscr{A},\mathbf{Ab}) \to \mathbf{Ab}$ maps an exact functor $E$ to $E(A)$. Therefore, there exists some $X_F \in \mathscr{A}$ such that $(F_f,-)|_{\mathrm{Ex}(\mathscr{A},\mathbf{Ab})}\cong \mathrm{ev}_{X_F}$. 
	
	Suppose $0 \to A \to B \to C \to 0$ is a short exact sequence in $\mathscr{A}$. As $E$ is an exact functor and the monoidal structure on $\mathscr{A}$ is exact in each variable,  \[0 \to E(A \otimes X_F) \to E(B \otimes X_F) \to E(C \otimes X_F) \to 0, \] is exact in $\mathbf{Ab}$. As a result, by the Yoneda lemma, \[0 \to ((A \otimes X_F  ,-),E) \to ((B \otimes X_F,-),E) \to ((C \otimes X_F,-),E) \to 0, \] is exact in $\mathbf{Ab}$ and by the adjunction isomorphism this gives the exact sequence \[0 \to ((X_F,-),\mathrm{hom}((A,-),E)) \to ((X_F,-),\mathrm{hom}((B,-),E)) \to ((X_F,-),\mathrm{hom}((C,-),E)) \to 0. \]  
	Applying the Yoneda lemma once more we have the exact sequence \[0 \to (\mathrm{hom}((A,-),E))(X_F) \to (\mathrm{hom}((B,-),E))(X_F) \to (\mathrm{hom}((C,-),E))(X_F) \to 0, \] which is isomorphic to \[0 \to (F_f,\mathrm{hom}((A,-),E)) \to (F_f,\mathrm{hom}((B,-),E)) \to (F_f,\mathrm{hom}((C,-),E)) \to 0, \] as we have already seen that $\mathrm{hom}((A,-),E)$, $\mathrm{hom}((B,-),E)$ and $\mathrm{hom}((C,-),E)$ are exact functors and $(F_f,-)|_{\mathrm{Ex}(\mathscr{A},\mathbf{Ab})}\cong \mathrm{ev}_{X_F}$. 
	
	Again, by the Yoneda lemma and adjunction isomorphisms we have for every $A \in \mathscr{A}$, $\mathrm{hom}(F_f,E)(A) \cong ((A,-),\mathrm{hom}(F_f,E)) \cong (F_f \otimes (A,-),E) \cong (F_f, \mathrm{hom}((A,-),E))$. What's more, all these isomorphisms are natural in $A$. Therefore \[ 0 \to \mathrm{hom}(F_f,E)(A) \to \mathrm{hom}(F_f,E)(B) \to \mathrm{hom}(F_f,E)(C) \to 0,\] is exact in $\mathbf{Ab}$ and $\mathrm{hom}(F_f,E)$ is an exact functor as required.
	
	((i) $\implies$ (iii)) Since we have shown that (i) $\iff$ (ii) we assume both hold.
	By (i) we have that $\mathrm{hom}((X,-),E)$ is exact for all $X \in \mathscr{A}$. Therefore, for all short exact sequences $0 \to A \to B \to C \to 0$ in $\mathscr{A}$, \[ 0 \to \mathrm{hom}((X,-),E)(A) \to \mathrm{hom}((X,-),E)(B) \to \mathrm{hom}((X,-),E)(C) \to 0,\] \noindent is exact. But we have isomorphisms $\mathrm{hom}((X,-),E)(A) \cong ((A,-),\mathrm{hom}((X,-),E)) \cong ((X,-) \otimes (A,-),E)=((X \otimes A,-),E) \cong E(X \otimes A)$ which are natural in $A$. Therefore, $0 \to E(X \otimes A) \to E(X \otimes B) \to E(X \otimes C) \to 0$ is also exact. 
	
	So for any exact sequence, $0 \to A \to B \to C \to 0$ in $\mathscr{A}$, $0 \to X \otimes A \to X \otimes B \to X \otimes C \to 0$ has exact image in $\mathbf{Ab}$ under any exact functor $E:\mathscr{A} \to \mathbf{Ab}$. By Lemma \ref{lem monoidal equivalence}, it remains to show that the equivalence $\mathscr{A} \simeq \mathrm{fun}(\mathrm{Ex}(\mathscr{A},\mathbf{Ab}))$ is monoidal.  
	
	By (ii), $\mathsf{S}_{\mathrm{Ex}}$ is a tensor ideal of $(\mathscr{A}\hbox{-}\mathrm{mod},\mathbf{Ab})^{\mathrm{fp}}$. Therefore we can define a monoidal structure on $\mathrm{fun}(\mathrm{Ex}(\mathscr{A},\mathbf{Ab}))$ (as in Definition \ref{remark strict}) such that the localisation functor $q:(\mathscr{A}\hbox{-}\mathrm{mod},\mathbf{Ab})^{\mathrm{fp}} \to (\mathscr{A}\hbox{-}\mathrm{mod},\mathbf{Ab})^{\mathrm{fp}}/\mathsf{S}_{\mathrm{Ex}} \simeq \mathrm{fun}(\mathrm{Ex}(\mathscr{A},\mathbf{Ab}))$ is a monoidal functor. 
	
	Note that the functor $\mathcal{Y}^2:\mathscr{A} \to (\mathscr{A}\hbox{-}\mathrm{mod},\mathbf{Ab})^{\mathrm{fp}}$ given by $A \mapsto ((A,-),-)$ is monoidal with respect to Day convolution product and the equivalence $\mathscr{A} \simeq \mathrm{fun}(\mathrm{Ex}(\mathscr{A},\mathbf{Ab}))$ from (\cite{PR10}, Theorem 2.2) can be taken to be $q \circ \mathcal{Y}^2$. Therefore, this equivalence is monoidal. Lemma \ref{lem monoidal equivalence} completes the proof.   
\end{proof}

\begin{remark} \label{Rmk exactness}
	Recall that the objects of the 2-category $\mathbb{ABEX}^{\otimes}$ are skeletally small abelian categories with  additive symmetric monoidal structures which are \textit{exact} in each variable. However, in most examples (for instance $\mathscr{A}=R\hbox{-}\mathrm{mod}$ for $R$ a commutative ring) the monoidal structure is only right exact. Theorem \ref{TFAE SE tensor ideal} shows where the equivalence fails without the exactness assumption. Indeed, if we desire the equivalence $\mathscr{A} \simeq \mathrm{fun}(\mathrm{Ex}(\mathscr{A},\mathbf{Ab}))$ to be monoidal, the Serre subcategory $\mathsf{S}_{\mathrm{Ex}}$ must be a tensor-ideal, to induce a monoidal structure on $\mathrm{fun}(\mathrm{Ex}(\mathscr{A},\mathbf{Ab}))$. 
\end{remark}

\begin{prop}
	There exists a 2-functor $\xi:\mathbb{ABEX}^{\otimes} \to \mathbb{DEF}^{\otimes}$ given on objects by $\mathscr{A} \mapsto (\mathrm{Ex}(\mathscr{A},\mathbf{Ab}), \mathscr{A}\hbox{-}\mathrm{Mod},\otimes)$ where the monoidal structure on $\mathscr{A}\hbox{-}\mathrm{Mod}$ is induced by Day convolution product.
\end{prop}

\begin{proof}
	By Theorem \ref{TFAE SE tensor ideal}, $\mathrm{Ex}(\mathscr{A},\mathbf{Ab})$ is an fp-hom-closed definable subcategory of $\mathscr{A}\hbox{-}\mathrm{Mod}$. Furthermore, as noted in the proof of Theorem \ref{TFAE SE tensor ideal}, $\mathscr{A} \simeq \mathrm{fun}(\mathrm{Ex}(\mathscr{A},\mathbf{Ab}))$ is a monoidal equivalence meaning the monoidal structure on $\mathrm{fun}(\mathrm{Ex}(\mathscr{A},\mathbf{Ab}))$ is exact. In turn this implies, by Proposition \ref{prop exactness only if}, that $\mathrm{Ex}(\mathscr{A},\mathbf{Ab})$ satisfies the exactness criterion. Therefore $\xi$ is well defined on objects. 
	
	Next we need to show that, given a morphism $E:\mathscr{A} \to \mathscr{B}$ in $\mathbb{ABEX}^{\otimes}$, $E^*:\mathrm{Ex}(\mathscr{B},\mathbf{Ab}) \to \mathrm{Ex}(\mathscr{A},\mathbf{Ab})$ is a morphism in $\mathbb{DEF}^{\otimes}$ that is $(E^*)_0:\mathrm{fun}(\mathrm{Ex}(\mathscr{A},\mathbf{Ab})) \to \mathrm{fun}(\mathrm{Ex}(\mathscr{B},\mathbf{Ab}))$ is monoidal. 
	
	 By the original anti-equivalence in \cite{PR10}, we have the following commutative diagram.
	 
	 \begin{tikzpicture}
	 \matrix (m) [matrix of math nodes,row sep=5em,column sep=4em,minimum width=2em]
	 {
	 	\mathscr{A} &  \mathrm{fun}(\mathrm{Ex}(\mathscr{A},\mathbf{Ab}))   \\
	 	\mathscr{B} & \mathrm{fun}(\mathrm{Ex}(\mathscr{B},\mathbf{Ab})) \\};
	 \path[-stealth]
	 (m-1-1) edge node [above] {$\sim$} (m-1-2)
	 (m-1-1) edge node [left] {$E$} (m-2-1)
	 (m-1-2) edge node [right] {$(E^*)_0$} (m-2-2)
	 (m-2-1) edge node [below] {$\sim$} (m-2-2);
	 \end{tikzpicture}
	 
	 By the proof of Theorem \ref{TFAE SE tensor ideal}, the equivalence given by the horizontal maps is monoidal. Therefore the inverse equivalence $\mathrm{fun}(\mathrm{Ex}(\mathscr{A},\mathbf{Ab})) \to \mathscr{A}$ is also monoidal and $(E^*)_0$ is naturally isomorphic to a monoidal functor, hence monoidal.
	 
	 Finally, $\xi$ acts on natural transformations in the same way as the original anti-equivalence, (forgetting the monoidal structure) and therefore is a well-defined 2-functor.
\end{proof}

\subsection{Completing the proof}

It remains to prove the following proposition.

\begin{prop}
For any $\mathscr{A} \in \mathbb{ABEX}^{\otimes}$ the functor $\epsilon_{\mathscr{A}}: \mathscr{A} \to \mathrm{fun}(\mathrm{Ex}(\mathscr{A},\mathbf{Ab}))$ given by $\epsilon_{\mathscr{A}}(A)=\mathrm{ev}_A$ is monoidal. Here $\mathrm{ev}_A:\mathrm{Ex}(\mathscr{A},\mathbf{Ab}) \to  \mathbf{Ab}$ maps an exact functor $F:\mathscr{A} \to \mathbf{Ab}$ to $F(A)$.

Similarly, for any $(\mathcal{D},\mathcal{C},\otimes) \in \mathbb{DEF}^{\otimes}$ the functor $\epsilon_{\mathcal{D}}:\mathcal{D} \to \mathrm{Ex}(\mathrm{fun}(\mathcal{D}),\mathrm{Ab})$ (as in the proof of Theorem 2.3 in \cite{PR10}) is a morphism in $\mathbb{DEF}^{\otimes}$.
\end{prop}
 
\begin{proof}
	By (\cite{P11}, Lemma 12.9 and Theorem 12.10) the functor \[(\mathscr{A}\hbox{-}\mathrm{mod},\mathbf{Ab})^{\mathrm{fp}} \xrightarrow{q} (\mathscr{A}\hbox{-}\mathrm{mod},\mathbf{Ab})^{\mathrm{fp}}/\mathsf{S}_{\mathrm{Ex}} \simeq \mathrm{fun}(\mathrm{Ex}(\mathscr{A},\mathbf{Ab})),\] maps a finitely presented functor $F:\mathscr{A}\hbox{-}\mathrm{mod} \to \mathbf{Ab}$ to $\overrightarrow{F}|_{\mathcal{D}}$ that is the restriction to $\mathcal{D}$ of the unique direct limit extension of $F$. By the Yoneda lemma $\epsilon_{\mathscr{A}}: \mathscr{A} \to \mathrm{fun}(\mathrm{Ex}(\mathscr{A},\mathbf{Ab}))$ is naturally equivalent to the functor  
	\[\mathscr{A} \xrightarrow{\mathcal{Y}^2} (\mathscr{A}\hbox{-}\mathrm{mod},\mathbf{Ab})^{\mathrm{fp}} \xrightarrow{q} (\mathscr{A}\hbox{-}\mathrm{mod},\mathbf{Ab})^{\mathrm{fp}}/\mathsf{S}_{\mathrm{Ex}} \simeq \mathrm{fun}(\mathrm{Ex}(\mathscr{A},\mathbf{Ab})),\] where $\mathcal{Y}^2:\mathscr{A} \to (\mathscr{A}\hbox{-}\mathrm{mod},\mathbf{Ab})^{\mathrm{fp}}$ denotes the Yoneda embedding $A \mapsto ((A,-),-)$. Therefore as the Yoneda embedding is monoidal with respect to Day convolution product and the monoidal structure on $\mathrm{fun}(\mathrm{Ex}(\mathscr{A},\mathbf{Ab}))$ is defined such that the localisation functor $q$ and the equivalence $(\mathscr{A}\hbox{-}\mathrm{mod},\mathbf{Ab})^{\mathrm{fp}} / \mathsf{S}_{\mathrm{Ex}} \simeq \mathrm{fun}(\mathrm{Ex}(\mathscr{A},\mathbf{Ab}))$ are monoidal, $\epsilon_{\mathscr{A}}$ is a monoidal functor. 	
	
	Next we show that, for all $(\mathcal{D},\mathcal{C},\otimes)$ in $\mathbb{DEF}^{\otimes}$, $(\epsilon_{\mathcal{D}})_0:\mathrm{fun}(\mathrm{Ex}(\mathrm{fun}(\mathcal{D}),\mathrm{Ab})) \to \mathrm{fun}(\mathcal{D})$ is monoidal. By \cite{PR10}, $\epsilon_{\mathrm{fun}(\mathcal{D})}:\mathrm{fun}(\mathcal{D})\to \mathrm{fun}(\mathrm{Ex}(\mathrm{fun}(\mathcal{D}),\mathrm{Ab}))$ is an equivalence so we have a functor, $\gamma:\mathrm{fun}(\mathrm{Ex}(\mathrm{fun}(\mathcal{D}),\mathrm{Ab})) \xrightarrow{\sim} \mathrm{fun}(\mathcal{D})$, which is both right and left adjoint to $\epsilon_{\mathrm{fun}(\mathcal{D})}$. We show that $(\epsilon_{\mathcal{D}})_0$ is naturally isomorphic to $\lambda$. 
	
	The unit of the adjunction $\gamma \dashv \epsilon_{\mathrm{fun}(\mathcal{D})}$ gives a natural isomorphism $\eta: \mathrm{Id}_{\mathrm{fun}(\mathrm{Ex}(\mathrm{fun}(\mathcal{D}),\mathbf{Ab}))} \xrightarrow{\sim}  \epsilon_{\mathrm{fun}(\mathcal{D})} \circ \gamma$.	
		
	Now, for $X \in \mathcal{D}$ and $F \in \mathrm{fun}(\mathrm{Ex}(\mathrm{fun}(\mathcal{D}),\mathrm{Ab}))$, $(\epsilon_{\mathcal{D}})_0((\epsilon_{\mathrm{fun}(\mathcal{D})} \circ \gamma)(F))(X)= \mathrm{ev}_{(\gamma)(F)}(\mathrm{ev}_{X})=\mathrm{ev}_{X}(\gamma(F))=\gamma(F)(X)$, so $(\epsilon_{\mathcal{D}})_0 \circ \epsilon_{\mathrm{fun}(\mathcal{D})} \circ \gamma=\gamma$. Therefore the composition of the natural isomorphism $\eta$ and the functor $(\epsilon_{\mathcal{D}})_0$ gives a natural isomorphism  	
	\[ (\epsilon_{\mathcal{D}})_0 \eta:(\epsilon_{\mathcal{D}})_0 \to (\epsilon_{\mathcal{D}})_0 \circ \epsilon_{\mathrm{fun}(\mathcal{D})} \circ \gamma=\gamma. \]	

 We have already seen that $\epsilon_{\mathrm{fun}(\mathcal{D})}$ is monoidal and therefore we can take $\gamma$ to also be monoidal (e.g. see \cite{EGNO15}, Remark 1.5.3).
 
Therefore $(\epsilon_{\mathcal{D}})_0$ is naturally isomorphic to a monoidal functor and so is itself a monoidal functor. Hence $\epsilon_{\mathcal{D}}:\mathcal{D} \to \mathrm{Ex}(\mathrm{fun}(\mathcal{D},\mathbf{Ab}))$ is a morphism in $\mathbb{DEF}^{\otimes}$ as required.     
\end{proof}

\begin{remark}
The following diagram commutes, where the 2-functors denoted by $\mathscr{F}$ are the forgetful 2-functors and the vertical maps are the 2-category anti-equivalences. 
	
	 \begin{tikzpicture}
	\matrix (m) [matrix of math nodes,row sep=5em,column sep=4em,minimum width=2em]
	{
		\mathbb{ABEX}^{\otimes} &  \mathbb{ABEX}   \\
		\mathbb{DEF}^{\otimes} & \mathbb{DEF} \\};
	\path[-stealth]
	(m-1-1) edge node [above] {$\mathscr{F}$} (m-1-2)
	(m-1-1) edge node [left] {Theorem 		\ref{Thm 2-cat eq}} (m-2-1)
	(m-2-1) edge (m-1-1)
	(m-1-2) edge node [right] {\cite{PR10}} (m-2-2)
	(m-2-2) edge (m-1-2)
	(m-2-1) edge node [below] {$\mathscr{F}$} (m-2-2);
	\end{tikzpicture}
	
\end{remark}

\section{Removing the exactness criterion} \label{without exactness}

As noted in Remark \ref{Rmk exactness}, for our 2-category anti-equivalence to hold, we required the monoidal structure on the skeletally small abelian category to be exact in each variable. However, given any fp-hom-closed definable subcategory $\mathcal{D}$ of a finitely accessible category $\mathcal{C}$,  which satisfies Assumption \ref{Rmk C}, we can induce a right exact monoidal structure on $\mathrm{fun}(\mathcal{D})$ as in Definition \ref{remark strict}. In many cases, this monoidal structure on the functor category is not left exact. In this section we consider what can be said about the relationship between definability and the monoidal structure for fixed $\mathcal{C}$, when we remove the exactness assumption.

\subsection{The Ziegler spectrum} \label{section ziegler}

In this section we define a coarser topology, $\mathrm{Zg}^{\mathrm{hom}}(\mathcal{C})$, on $\mathrm{pinj}_{\mathcal{C}}$ such that the identity morphism $\mathrm{Zg}(\mathcal{C}) \to \mathrm{Zg}^{\mathrm{hom}}(\mathcal{C})$ is a continous map.

\begin{theorem}
	Setting the closed subcategories of $\mathrm{pinj}_{\mathcal{C}}$ to be those given by the indecomposable pure-injectives contained in an fp-hom-closed definable subcategory of $\mathcal{C}$ defines a topology on $\mathrm{pinj}_{\mathcal{C}}$ which we will call the \textbf{fp-hom-closed Ziegler topology} and denote by $\mathrm{Zg}(\mathcal{C})^{\mathrm{hom}}$.  
\end{theorem}

\begin{proof}
	We must show that a finite union and arbitrary intersection of closed subcategories is closed. Abusing notation slightly, we will write $\mathcal{D} \cap \mathrm{pinj}_{\mathcal{C}}$ for the isomorphism classes of indecomposable pure-injective objects contained in $\mathcal{D}$, that is the closed subset of the Ziegler spectrum corresponding to $\mathcal{D}$. 
	
	We know (since the Ziegler spectrum defines a topology e.g (\cite{P11}, Theorem 14.1)) that given two definable subcategories $\mathcal{D}$ and $\mathcal{D'}$, the definable subcategory generated by their union, $\big{<} \mathcal{D} \cup \mathcal{D}'\big{>}^{\mathrm{def}}$, satisfies \[\big{<} \mathcal{D} \cup \mathcal{D}'\big{>}^{\mathrm{def}} \cap \mathrm{pinj}_{\mathcal{C}}=(\mathcal{D} \cap \mathrm{pinj}_{\mathcal{C}}) \cup (\mathcal{D}' \cap \mathrm{pinj}_{\mathcal{C}}).\] We must show that, if $\mathcal{D}$ and $\mathcal{D}'$ are fp-hom-closed, then so is $\big{<} \mathcal{D} \cup \mathcal{D}'\big{>}^{\mathrm{def}}$. Notice that the Serre subcategory corresponding to $\big{<} \mathcal{D} \cup \mathcal{D}'\big{>}^{\mathrm{def}}$ is given by the intersection of the Serre subcategories corresponding to $\mathcal{D}$ and $\mathcal{D'}$, say $\mathsf{S}_{\mathcal{D}}$ and $\mathsf{S}_{\mathcal{D}'}$ respectively. By Theorem \ref{thm},  $\mathsf{S}_{\mathcal{D}}$ and $\mathsf{S}_{\mathcal{D}'}$ are tensor-ideals so $\mathsf{S}_{\mathcal{D}} \cap \mathsf{S}_{\mathcal{D}'}$ must also be a tensor-ideal. Applying Theorem \ref{thm} again gives that $\big{<} \mathcal{D} \cup \mathcal{D}'\big{>}^{\mathrm{def}}$ is fp-hom-closed. It is straightforward to see that the intersection of fp-hom-closed definable subcategories is fp-hom-closed and this completes the proof. 
\end{proof}

Thus we have the following tensor-analogue of Theorem \ref{loc fp def Serre}.

\begin{cor}
	Let $\mathcal{C}$ be as in Assumption \ref{Rmk C}. There is a bijection between:
	\begin{enumerate}[label=(\roman*)]
		\item the fp-hom-closed definable subcategories of $\mathcal{C}$,
		
		\item the Serre tensor-ideals of $\mathcal{C}^{\mathrm{fp}}\hbox{-}\mathrm{mod}$,
		
		\item the closed subsets of $\mathrm{Zg}(\mathcal{C})^{\mathrm{hom}}$. 	
	\end{enumerate}	
\end{cor}  
		
\subsection{A rigidity assumption} \label{section rigid}

Next we move on to the context where $\mathcal{C}^{\mathrm{fp}}$ forms a rigid monoidal subcategory. In this setting, we get the following corollary to Theorem \ref{thm}, giving a definable tensor-ideal/Serre tensor-ideal correspondence.

\begin{cor} \label{Cfp rigid case}
	Let $\mathcal{C}$ be a finitely accessible category with products and suppose that $(\mathcal{C},\otimes, 1)$ is a closed symmetric monoidal category such that $\mathcal{C}^{\mathrm{fp}}$ is a symmetric rigid monoidal subcategory. Let $\mathsf{S}$ be a Serre subcategory of $\mathcal{C}^{\mathrm{fp}}\hbox{-}\mathrm{mod}$ and let $\mathcal{D}$ be the corresponding definable subcategory of $\mathcal{C}$ as in (Theorem \ref{loc fp def Serre}).
	
	Then, $\mathsf{S}$ is a Serre tensor-ideal of $\mathcal{C}^{\mathrm{fp}}\hbox{-}\mathrm{mod}$  with respect to the induced tensor product if and only if $\mathcal{D}$ is a definable tensor-ideal of $\mathcal{C}$.
\end{cor}

\begin{proof}
	By Theorem \ref{thm} we have that $\mathsf{S}$ is a Serre tensor-ideal if and only if $\mathcal{D}$ is fp-hom-closed. By rigidity of $\mathcal{C}^{\mathrm{fp}}$, there exists a natural equivalence $\mathrm{hom}(A,-) \cong A^{\vee} \otimes -$ for all every $A \in \mathcal{C}^{\mathrm{fp}}$, therefore $\mathcal{D}$ is fp-hom-closed if and only if it is closed under tensoring with objects of $\mathcal{C}^{\mathrm{fp}}$. Suppose $X \in \mathcal{C}$ and $D \in \mathcal{D}$. As $\mathcal{C}$ is finitely accessible we can write $X$ as a direct limit $X=\underrightarrow{\mathrm{lim}}_{i \in I} X_i$ where the $X_i$ are finitely presented.  Therefore, if $\mathcal{D}$ is closed under tensoring with objects of $\mathcal{C}^{\mathrm{fp}}$, then $X \otimes D \cong (\underrightarrow{\mathrm{lim}}_{i \in I} X_i) \otimes D \cong \underrightarrow{\mathrm{lim}}_{i \in I} (X_i \otimes D) \in \mathcal{D}$, as $- \otimes D$ commutes with direct limits and $\mathcal{D}$ is closed under direct limits. 
\end{proof}

\subsection{Elementary duality} \label{section duality}

Throughout this section assume $\mathcal{A}$ is a small preadditive category with an additive rigid monoidal structure. We show that elementary duality (see Proposition \ref{prop background def dual}) maps fp-hom-closed definable subcategories of $\mathrm{Mod}\hbox{-}\mathcal{A}$ to definable tensor-ideals of $\mathcal{A}\hbox{-}\mathrm{Mod}$. 

\begin{notation}
	We will denote the monoidal structure on $\mathcal{A}$ by $\otimes$, while $\otimes_{\mathcal{A}}$ denotes the tensor product of $\mathcal{A}$-modules given in Definition \ref{def tensor over A}.
\end{notation} 

\begin{definition}
	Given a finitely presented right $\mathcal{A}$-module $M \in \mathrm{mod}\hbox{-}\mathcal{A}$ with presentation \[(-,m_1) \xrightarrow{(-,m)} (-,m_2) \to M \to 0\] where $m:m_1 \to m_2$ is a morphism in $\mathcal{A}$, define (up to isomorphism) the finitely presented left $\mathcal{A}$-module $M^d \in \mathcal{A}\hbox{-}\mathrm{mod}$ to have presentation \[(m_1^{\vee},-) \xrightarrow{(m^{\vee},-)} (m_2^{\vee},-) \to M^d \to 0,\] where $m^{\vee}:m_2^{\vee} \to m_1^{\vee}$ is the dual morphism to $m$ in $\mathcal{A}$.
	
	Similarly, given a finitely presented left $\mathcal{A}$-module $N \in \mathcal{A}\hbox{-}\mathrm{mod}$ with presentation \[(n_2,-) \xrightarrow{(n,-)} (n_1,-) \to N \to 0\] where $n:n_1 \to n_2$ is a morphism in $\mathcal{A}$, define (up to isomorphism) the finitely presented right $\mathcal{A}$-module $N^d \in \mathrm{mod}\hbox{-}\mathcal{A}$ to have presentation \[(-,n_2^{\vee}) \xrightarrow{(-,n^{\vee})} (-,n_1^{\vee}) \to N^d \to 0,\] where $n^{\vee}:n_2^{\vee} \to n_1^{\vee}$ is the dual morphism to $n$ in $\mathcal{A}$.  
\end{definition}

\begin{prop} \label{prop duality}
	Let $\mathcal{A}$ be a small preadditive category with an additive symmetric rigid monoidal structure and induce monoidal structures on $\mathcal{A}\hbox{-}\mathrm{Mod}$ and $\mathrm{Mod}\hbox{-}\mathcal{A}$ via Day convolution product.
	
	The maps $(-)^d:\mathcal{A}\hbox{-}\mathrm{mod} \leftrightarrow \mathrm{mod}\hbox{-}\mathcal{A}$ give an equivalence between $\mathcal{A}\hbox{-}\mathrm{mod}$ and $\mathrm{mod}\hbox{-}\mathcal{A}$.  
\end{prop}

\begin{proof}
	Fix a presentation for each $N \in \mathcal{A}\hbox{-}\mathrm{mod}$. First let us show that $(-)^d:\mathcal{A}\hbox{-}\mathrm{mod} \to \mathrm{mod}\hbox{-}\mathcal{A}$ is functorial. Suppose $h:N \to N'$ is a morphism in $\mathcal{A}\hbox{-}\mathrm{mod}$ where $N$ and $N'$ have presentations \[ (n_2,-) \xrightarrow{(n,-)} (n_1,-) \to N \to 0 \] and  \[ (n'_2,-) \xrightarrow{(n',-)} (n'_1,-) \to N' \to 0 \]  respectively. 
	
	By projectivity of representables we can choose $h_1:n'_1 \to n_1$ and $h_2:n'_2 \to n_2$ such that the following diagram commutes.
	
	\begin{tikzpicture}
	\matrix (m) [matrix of math nodes,row sep=4em,column sep=4em,minimum width=2em]
	{
		(n_2,-) & (n_1,-) & N & 0  \\
		(n'_2,-) & (n'_1,-) & N' & 0 \\};
	\path[-stealth]
	(m-1-1) edge node [above] {$(n,-)$} (m-1-2)
	(m-1-2) edge (m-1-3)
	(m-1-3) edge (m-1-4)
	(m-2-1) edge node [below] {$(n',-)$} (m-2-2)
	(m-2-2) edge (m-2-3)
	(m-2-3) edge (m-2-4)
	
	(m-1-1) edge node [left] {$(h_2,-)$} (m-2-1)
	(m-1-2) edge node [left] {$(h_1,-)$} (m-2-2)
	(m-1-3) edge node [right] {$h$} (m-2-3);
	\end{tikzpicture}
	
	Thus, $n \circ h_1=h_2 \circ n'$ and dualising we get $h_1^{\vee} \circ n^{\vee}=n'^{\vee} \circ h_2^{\vee}$. Therefore we have the following commutative diagram where the map $h^d$ is uniquely determined.
	
	\begin{tikzpicture}
	\matrix (m) [matrix of math nodes,row sep=4em,column sep=4em,minimum width=2em]
	{
		(-,{n_2}^{\vee}) & (-,{n_1}^{\vee}) & N^d & 0  \\
		(-,{n'_2}^{\vee}) & (-,{n'_1}^{\vee}) & N'^d & 0 \\};
	\path[-stealth]
	(m-1-1) edge node [above] {$(-,n^{\vee})$} (m-1-2)
	(m-1-2) edge (m-1-3)
	(m-1-3) edge (m-1-4)
	(m-2-1) edge node [below] {$(-,n'^{\vee})$} (m-2-2)
	(m-2-2) edge (m-2-3)
	(m-2-3) edge (m-2-4)
	
	(m-1-1) edge node [left] {$(-,h_2^{\vee})$} (m-2-1)
	(m-1-2) edge node [left] {$(-,h_1^{\vee})$} (m-2-2)
	(m-1-3) edge [dashed] node [right] {$h^d$} (m-2-3);
	\end{tikzpicture}	 
	
	It is straightforward to check that any choice of $h_1$ and $h_2$ induce the same map $h^d$ and functoriality of $(-)^{\vee}:\mathcal{A} \to \mathcal{A}$ implies functoriality of $(-)^d$. So (given a choice of presentation for all $N \in \mathcal{A}\hbox{-}\mathrm{mod}$) we have a well-defined functor, $(-)^d:\mathcal{A}\hbox{-}\mathrm{mod} \to \mathrm{mod}\hbox{-}\mathcal{A}$. Furthermore, since we have a natural isomorphism $1_{\mathcal{A}} \to ((-)^{\vee})^{\vee}$, by construction,  the functor $(-)^d:\mathrm{mod}\hbox{-}\mathcal{A} \to \mathcal{A}\hbox{-}\mathrm{mod}$ (fixing a presentation for each $N \in \mathcal{A}\hbox{-}\mathrm{mod}$) clearly gives a quasi-inverse.  
\end{proof}

\begin{lemma} \label{lemma otimes A Md}
	For every $M \in \mathrm{mod}\hbox{-}\mathcal{A}$, $L \in \mathrm{Mod}\hbox{-}\mathcal{A}$ and $N \in \mathcal{A}\hbox{-}\mathrm{Mod}$, we have an isomorphism \[ (L \otimes M) \otimes_{\mathcal{A}} N \cong L \otimes_{\mathcal{A}} (M^d \otimes N),\] natural in $L$ and $N$. 
\end{lemma}

\begin{proof}
	First let us prove that for every $a \in \mathcal{A}$, we have an isomorphism $(L \otimes (-,a)) \otimes_{\mathcal{A}} N \cong L \otimes_{\mathcal{A}} ((a^{\vee}, -) \otimes N)$ which is natural in $N$ and $L$. Recall that $-\otimes_{\mathcal{A}} N$ and $- \otimes (-,a)$ are both right exact and therefore preserve direct limits. Therefore, as $\mathrm{Mod}\hbox{-}\mathcal{A}$ is locally finitely presentable, it is sufficient to assume that $L$ is finitely presented. Suppose $L$ has presentation $(-,l_1) \xrightarrow{(-,l)} (-,l_2) \to L \to 0$. By right exactness of $-\otimes_{\mathcal{A}} N$ and $- \otimes (-,a)$ we have an exact sequence \[ (-,l_1 \otimes a) \otimes_{\mathcal{A}} N \xrightarrow{(-,l \otimes a) \otimes_{\mathcal{A}} N} (-,l_2 \otimes a) \otimes_{\mathcal{A}} N \to (L \otimes (-,a)) \otimes_{\mathcal{A}} N \to 0. \]  By definition of $\otimes_{\mathcal{A}}$, $(-,l \otimes a) \otimes_{\mathcal{A}} N:(-,l_1 \otimes a) \otimes_{\mathcal{A}} N \to (-,l_2 \otimes a) \otimes_{\mathcal{A}} N$ is given by $N(l \otimes a):N(l_1 \otimes a) \to N(l_2 \otimes a)$. Thus by the Yoneda lemma we have the following commutative diagram in $\mathbf{Ab}$.
	
	\begin{tikzpicture}
	\matrix (m) [matrix of math nodes,row sep=4em,column sep=6em,minimum width=2em]
	{
		(-,l_1 \otimes a) \otimes_{\mathcal{A}} N & (-,l_2 \otimes a) \otimes_{\mathcal{A}} N \\
		((l_1 \otimes a,-),N) & ((l_2 \otimes a,-),N) \\};
	\path[-stealth]
	(m-1-1) edge node [above] {$(-,l \otimes a) \otimes_{\mathcal{A}} N$} (m-1-2)
	(m-2-1) edge node [below] {$((l \otimes a,-),N)$} (m-2-2)
	(m-1-1) edge node [left] {$\cong$} (m-2-1)
	(m-1-2) edge node [left] {$\cong$} (m-2-2);
	\end{tikzpicture}
	
	By considering Lemma \ref{(x,-) otimes - exact}, we see that $(a^{\vee},-) \otimes -:\mathcal{A}\hbox{-}\mathrm{Mod} \to \mathcal{A}\hbox{-}\mathrm{Mod}$ is right adjoint to $(a,-) \otimes -:\mathcal{A}\hbox{-}\mathrm{Mod} \to \mathcal{A}\hbox{-}\mathrm{Mod}$. Thus we have the following commutative diagram where the first row of downwards arrows is given by the adjointness isomorphisms and the second row is given by the Yoneda lemma.

	\begin{tikzpicture}
	\matrix (m) [matrix of math nodes,row sep=4em,column sep=8em,minimum width=2em]
	{	((l_1 \otimes a,-),N) & ((l_2 \otimes a,-),N) \\
		((l_1,-),(a^{\vee},-) \otimes N) & ((l_2,-),(a^{\vee},-) \otimes N) \\
		((a^{\vee},-) \otimes N)(l_1) & ((a^{\vee},-) \otimes N)(l_2) \\};
	\path[-stealth]
	(m-1-1) edge node [above] {$((l \otimes a,-),N)$} (m-1-2)
	(m-2-1) edge node [above] {$((l,-), (a^{\vee},-) \otimes N)$} (m-2-2)
	(m-1-1) edge node [left] {$\cong$} (m-2-1)
	(m-1-2) edge node [left] {$\cong$} (m-2-2)
	(m-3-1) edge node [below] {$((a^{\vee},-) \otimes N)(l)$} (m-3-2)
	(m-2-1) edge node [left] {$\cong$} (m-3-1)
	(m-2-2) edge node [left] {$\cong$} (m-3-2);
	\end{tikzpicture}
	
	By the definition of $\otimes_{\mathcal{A}}$ we have $((a^{\vee},-) \otimes N)(l)=(-,l) \otimes_{\mathcal{A}} ((a^{\vee},-) \otimes N)$. Furthermore, by right exactness of $- \otimes_{\mathcal{A}} ((a^{\vee},-) \otimes N)$ we have exact sequence \[(-,l_1) \otimes_{\mathcal{A}} ((a^{\vee},-) \otimes N) \xrightarrow{(-,l) \otimes_{\mathcal{A}} ((a^{\vee},-) \otimes N)} (-,l_2) \otimes_{\mathcal{A}} ((a^{\vee},-) \otimes N) \to L \otimes_{\mathcal{A}} ((a^{\vee},-) \otimes N) \to 0.\] Thus we have an induced isomorphism $(L \otimes (-,a)) \otimes_{\mathcal{A}} N \cong L \otimes_{\mathcal{A}} ((a^{\vee}, -) \otimes N)$ as shown on the commutative diagram below.  
	
	\noindent \begin{tikzpicture}
	\matrix (m) [matrix of math nodes,row sep=4em,column sep=2em,minimum width=2em]
	{		(-,l_1 \otimes a) \otimes_{\mathcal{A}} N & & (-,l_2 \otimes a) \otimes_{\mathcal{A}} N & (L \otimes (-,a)) \otimes_{\mathcal{A}} N & 0 \\
		((l_1 \otimes a,-),N) & & ((l_2 \otimes a,-),N) & & \\
		((l_1,-),(a^{\vee},-) \otimes N) & & ((l_2,-),(a^{\vee},-) \otimes N) & & \\
		(-,l_1) \otimes_{\mathcal{A}} 	((a^{\vee},-) \otimes N) & & (-,l_2) \otimes_{\mathcal{A}} ((a^{\vee},-) \otimes N) & L \otimes_{\mathcal{A}} ((a^{\vee},-) \otimes N) & 0 \\};
	\path[-stealth]
	(m-1-1) edge node [above] {$(-,l \otimes a) \otimes_{\mathcal{A}} N$} (m-1-3)
	(m-1-3) edge (m-1-4)
	(m-1-4) edge (m-1-5)
	(m-2-1) edge node [above] {$((l \otimes a,-),N)$} (m-2-3)
	(m-3-1) edge node [above] {$((l,-), (a^{\vee},-) \otimes N)$} (m-3-3)
	(m-1-1) edge node [left] {$\cong$} (m-2-1)
	(m-1-3) edge node [left] {$\cong$} (m-2-3)
	(m-4-1) edge node [above] {$(-,l) \otimes_{\mathcal{A}} ((a^{\vee},-) \otimes N)$} (m-4-3)
	(m-2-1) edge node [left] {$\cong$} (m-3-1)
	(m-2-3) edge node [left] {$\cong$} (m-3-3)
	(m-3-1) edge node [left] {$\cong$} (m-4-1)
	(m-3-3) edge node [left] {$\cong$} (m-4-3)
	(m-4-3) edge (m-4-4)
	(m-4-4) edge (m-4-5)
	(m-1-4) edge [dashed] (m-4-4);
	\end{tikzpicture}
	
	As each of the isomorphisms in the first and second columns are natural in $(-,l_i)$ and $N$, the induced isomorphism is natural in $L$ and $N$. Furthermore, by properties of dual morphisms in $\mathcal{A}$ we have that for every $m:m_1 \to m_2$ in $\mathcal{A}$ the following square commutes for i=1,2.

	\begin{tikzpicture}
	\matrix (m) [matrix of math nodes,row sep=4em,column sep=8em,minimum width=2em]
	{	((l_i \otimes m_1,-),N) & ((l_i \otimes m_2,-),N) \\
		((l_i,-),(m_1^{\vee},-) \otimes N) & ((l_i,-),(m_2^{\vee},-) \otimes N) \\};
	\path[-stealth]
	(m-1-1) edge node [above] {$((l_i \otimes m,-),N)$} (m-1-2)
	(m-2-1) edge node [above] {$((l_i,-), (m^{\vee},-) \otimes N)$} (m-2-2)
	(m-1-1) edge node [left] {$\cong$} (m-2-1)
	(m-1-2) edge node [left] {$\cong$} (m-2-2);
	\end{tikzpicture}  
	
	Therefore the induced isomorphisms $(L \otimes (-,m_i)) \otimes_{\mathcal{A}} N \cong L \otimes_{\mathcal{A}} ((m_i^{\vee}, -) \otimes N)$ for $i=1,2$ commute with any morphism $m:m_1 \to m_2$ in $\mathcal{A}$ in the following sense. 
	
	\noindent	 	 \begin{tikzpicture}
	\matrix (m) [matrix of math nodes,row sep=4em,column sep=4em,minimum width=2em]
	{	(L \otimes (-,m_1)) \otimes_{\mathcal{A}} N & (L \otimes (-,m_2))  \otimes_{\mathcal{A}} N  & (L \otimes M)  \otimes_{\mathcal{A}} N & 0 \\
		L   \otimes_{\mathcal{A}} ((m_1^{\vee},-) \otimes N) & L \otimes_{\mathcal{A}} ((m_2^{\vee},-)  \otimes N)  & L \otimes_{\mathcal{A}} (M^d \otimes N) & 0 \\};
	\path[-stealth]
	(m-1-1) edge node [above] {$(L \otimes (-,m)) \otimes_{\mathcal{A}} N$} (m-1-2)
	(m-2-1) edge node [above] {$L \otimes_{\mathcal{A}} ((m^{\vee},-) \otimes N)$} (m-2-2)
	(m-1-2) edge (m-1-3)
	(m-1-3) edge (m-1-4)
	(m-2-2) edge (m-2-3)
	(m-2-3) edge (m-2-4)
	(m-1-1) edge node [left] {$\cong$} (m-2-1)
	(m-1-2) edge node [left] {$\cong$} (m-2-2)
	(m-1-3) edge [dashed] node [left] {$\cong$} (m-2-3);
	\end{tikzpicture} 
	
	Hence the desired isomorphism $(L \otimes M) \otimes_{\mathcal{A}} N \cong L \otimes_{\mathcal{A}} (M^d \otimes N)$ is determined uniquely by the commutative diagram shown above. 
\end{proof}

\begin{theorem}
	Let $\mathcal{A}$ be a small preadditive category with an additive, symmetric, rigid, monoidal structure and induce monoidal structures on $\mathcal{A}\hbox{-}\mathrm{Mod}$ and $\mathrm{Mod}\hbox{-}\mathcal{A}$ via Day convolution product. 
	
	A definable subcategory $\mathcal{D} \subseteq \mathrm{Mod}\hbox{-}\mathcal{A}$ is fp-hom-closed if and only if the dual definable subcategory $\delta \mathcal{D} \subseteq \mathcal{A}\hbox{-}\mathrm{Mod}$ is a tensor-ideal. 
\end{theorem}

\begin{proof}
	By Theorem \ref{thm}, $\mathcal{D} \subseteq \mathrm{Mod}\hbox{-}\mathcal{A}$ is an fp-hom-closed definable subcategory if and only if the corresponding Serre subcategory $\mathsf{S} \subseteq (\mathrm{mod}\hbox{-}\mathcal{A},\mathrm{Ab})^{\mathrm{fp}}$ is a tensor-ideal. 
	
	By (\cite{BRB}, Proposition 10.3.5), every functor in the dual Serre subcategory $\delta \mathsf{S} \subseteq (\mathcal{A}\hbox{-}\mathrm{mod},\mathbf{Ab})^{\mathrm{fp}}$ has form $\delta F_f$ with presentation \[ 0 \to \delta F \to A \otimes_{\mathcal{A}} - \xrightarrow{f \otimes_{\mathcal{A}}-} B \otimes_{\mathcal{A}} -\] for some $F_f \in \mathsf{S}$. Therefore, $X \in \delta \mathcal{D}$ if and only if for every $f:A \to B$ in $\mathcal{A}$ such that $F_f \in \mathsf{S}$, $\delta F_f(X)=0$ equivalently, $f \otimes_{\mathcal{A}}X:A\otimes_{\mathcal{A}}X \to B \otimes_{\mathcal{A}}X$ is a monomorphism. 
	
	Thus $\mathcal{D}$ is fp-hom-closed if and only if $\delta\mathcal{D}$ satisfies the following. $X \in \delta \mathcal{D}$ if and only if for every $F_f \in \mathsf{S}$, and every $M \in \mathrm{mod}\hbox{-}\mathcal{A}$, \[ (M \otimes f) \otimes_{\mathcal{A}} X:(M \otimes A) \otimes_{\mathcal{A}} X \to (M \otimes B) \otimes_{\mathcal{A}} X\] is a monomorphism. But by Lemma \ref{lemma otimes A Md}, \[ (M \otimes f) \otimes_{\mathcal{A}} X:(M \otimes A) \otimes_{\mathcal{A}} X \to (M \otimes B) \otimes_{\mathcal{A}} X\] is a monomorphism if and only if \[  f \otimes_{\mathcal{A}} (M^d \otimes X): A \otimes_{\mathcal{A}} (M^d \otimes X) \to B \otimes_{\mathcal{A}} (M^d \otimes X)\] is a monomorphism. Therefore, $\mathcal{D}$ is fp-hom-closed if and only if for every $X \in \delta\mathcal{D}$ and all $M \in \mathrm{mod}\hbox{-}\mathcal{A}$, $M^d \otimes X \in \delta \mathcal{D}$ or equivalently for all $N \in \mathcal{A}\hbox{-}\mathrm{mod}$, $N \otimes X \in \delta \mathcal{D}$ as $(-)^d$ is an equivalence (see Proposition \ref{prop duality}). That is, $\mathcal{D}$ is a fp-hom-closed if and only if $\delta \mathcal{D}$ is closed under tensoring with finitely presented left $\mathcal{A}$-modules if and only if $\delta \mathcal{D}$ is a tensor-ideal, as required.    
\end{proof}

\section{Examples}	\label{section examples}

\subsection{Tensor product of $R$-modules}

Let us consider the case where $\mathcal{C}=R\hbox{-}\mathrm{Mod}$ for a commutative ring $R$. Here, $R\hbox{-} \mathrm{Mod}$ has a closed symmetric monoidal structure with tensor product given by $\otimes_R$. The tensor unit is $R$ and the internal hom-functor is given by the usual hom-set with $R$-module structure given by $(rf)(x)=rf(x)=f(rx)$ for all $x \in X$ where $f \in \mathrm{hom}(X,Y)=\mathrm{Hom}_R(X,Y)$ and $r \in R$. Note that $R\hbox{-}\mathrm{mod}$ is a monoidal subcategory. 

The next result shows that if a functor $F \in (R\hbox{-} \mathrm{mod}, \mathbf{Ab})^{fp}$ belongs to some Serre subcategory $\mathsf{S}$, and if $F$ is `simple enough' then $G \otimes_R F \in \mathsf{S}$ for any finitely presented functor $G$.   

\begin{prop} \label{pdim}
	Let $R$ be a commutative ring, $\mathsf{S} \subseteq (R\hbox{-} \mathrm{mod}, \mathbf{Ab})^{fp}$ be a Serre subcategory and $F \in \mathsf{S}$ satisfy $\mathrm{pdim}(F)=0$ or $\mathrm{pdim}(F)=1$. Then for any $G \in (R\hbox{-} \mathrm{mod}, \mathbf{Ab})^{fp}$, $G \otimes F \in \mathsf{S}$, where $\otimes$ denotes the tensor product induced by $\otimes_R$ on $R\hbox{-}\mathrm{Mod}$.
\end{prop}

\begin{proof}
	By Lemma \ref{definable tensor ideal respresentables}, we can take $G=(C,-)$. Throughout, let $\mathcal{D}$ be the definable subcategory associated to $\mathsf{S}$ as in Theorem \ref{loc fp def Serre}.
	
	Suppose $F \in \mathsf{S}$ satisfies $\mathrm{pdim}(F)=0$. Then $F=(A,-)$ for some $A \in R\hbox{-} \mathrm{mod}$.  Therefore, for all $D \in \mathcal{D}$, $(A,D)=0$. For any $C \in R\hbox{-} \mathrm{mod}$ we have $(C,-) \otimes (A,-)=(C \otimes_R A,-)$. We want to show that for all $D \in \mathcal{D}$, $(C \otimes_R A,D)=0$. But by the adjunction isomorphism we have $(C \otimes_R A,D)\cong (C,\mathrm{hom}(A,D))=(C,(A,D))=(C,0)=0$, so $(C,-) \otimes (A,-) \in \mathsf{S}$, as required.
	
	Now suppose $F \in \mathsf{S}$ satisfies $\mathrm{pdim}(F)=1$. Then we have an exact sequence
	\[ 0 \to (B,-) \xrightarrow{(f,-)} (A,-) \xrightarrow{\pi} F \to 0, \] where the map $f:A \to B$ is an epimorphism in $R\hbox{-}\mathrm{mod}$. We want to show that for all $C \in R\hbox{-}\mathrm{mod}$, $(C,-) \otimes F \in \mathsf{S}$ i.e. the map $(C \otimes_R B,D) \xrightarrow{(C \otimes_R f,D)} (C \otimes_R A,D)$ is an epimorphism for all $D \in \mathcal{D}$. 
	
	As $F \in \mathsf{S}$, $(B,D) \xrightarrow{(f,D)} (A,D)$ is an isomorphism, for every $D \in \mathcal{D}$. Therefore for any $C \in R\hbox{-}\mathrm{mod}$, the map $(C,(B,D)) \xrightarrow{(C,(f,D))} (C,(A,D))$ is an isomorphism and the tensor-hom adjunction gives the following commutative diagram. 
	
		\begin{tikzpicture}
	\matrix (m) [matrix of math nodes,row sep=5em,column sep=3em,minimum width=2em]
	{
		(C \otimes_R B, D)  &  (C \otimes_R A, D)  \\
		(C,(B,D)) &  (C,(A,D)) \\};
	\path[-stealth]
	(m-1-1) edge node [above] {$(C \otimes_R f,D)$} (m-1-2)
	(m-1-1) edge node [left] {$\cong$} (m-2-1)
	(m-1-2) edge node [right] {$\cong$} (m-2-2)
	(m-2-1) edge node [below] {$\cong$} (m-2-2);
	\end{tikzpicture}
	
	Therefore for any $C \in R\hbox{-}\mathrm{mod}$, the map $(C \otimes_R B,D) \xrightarrow{(C \otimes_R f,D)} (C \otimes_R A,D)$ is an epimorphism for all $D \in \mathcal{D}$ and $(C,-) \otimes F \in \mathsf{S}$ as required.	
\end{proof}

Proposition \ref{pdim} does not hold for $\mathrm{pdim}(F)=2$. Indeed, the Serre subcategory generated by $T$ in Example \ref{Mike otimesR example} given below provides a counter example.

The following example is from (\cite{P18}, Section 13).

\begin{example}  (\cite{P18}, Section 13) \label{Mike otimesR example}
	Let $R=k[\epsilon: \epsilon^2=0]$, where $k$ is any field. We can define a monoidal structure on the category $R\hbox{-} \mathrm{Mod}$ with $\otimes:R\hbox{-}\mathrm{Mod} \times R\hbox{-}\mathrm{Mod} \to R\hbox{-}\mathrm{Mod}$ given by the usual tensor product of $R$-modules, $\otimes=\otimes_R$. We extend this to a monoidal structure on $(R\hbox{-}\mathrm{mod},\mathbf{Ab})^{\mathrm{fp}}$ using Day convolution product. First note that the only indecomposable $R$-modules are $_RR$ and $U=R / \mathrm{rad}(R)=R / \big{<}\epsilon \big{>} \cong k$. In fact every $R$-module is isomorphic to a direct sum of copies of these simple modules. We have $R \otimes_R R \cong R$, $R \otimes_R U \cong U$ and $U \otimes_R U \cong U$.
	
	Consider the exact sequence $0 \to \big{<}\epsilon\big{>} \xrightarrow{j} R \xrightarrow{p} U \to 0.$ Let $S$ and $T$ be the functors such that we have exact sequences $0 \to (U,-) \xrightarrow{(p,-)} (R,-) \to S \to 0$ and 	$0 \to (U,-) \xrightarrow{(p,-)} (R,-) \xrightarrow{(j,-)} (\big{<} \epsilon \big{>},-) \to T \to 0$. By Section 13 of \cite{P18}, the indecomposable functors in $(R\hbox{-}\mathrm{mod},\mathbf{Ab})^{\mathrm{fp}}$ are
	
	$$S:M \mapsto \epsilon M,$$
		
	$$T: M \mapsto \mathrm{ann}_M(\epsilon)/ \epsilon M,$$
	
	$$(U,-):M \mapsto \mathrm{ann}_M(\epsilon),$$
	
	$$W:M \mapsto M /\epsilon M,$$ and
	
	$$(R,-):M \mapsto M.$$

	The table below shows the action of the tensor product on $(R\hbox{-}\mathrm{mod},\mathbf{Ab})^{\mathrm{fp}}$ (given in \cite{P18}, Section 13.3).
	
	\FloatBarrier	
	\begin{table}[htbp!]
		\begin{tabular}{l|lllll}
			$\otimes$                  & $S~~~$     & $T$        & $(U,-)~~~$ & $W$        & $(R,-)$                 \\ \hline
			& & & & & \\
			$S$                        & $S$        & 0          & 0          & $S$        & $S$                     \\
			& & & & & \\
			$T$                        & 0          & $(U,-)~~~$ & $(U,-)$    & $T$        & $T$                     \\
			& & & & & \\
			$(U,-)$                    & 0          & $(U,-)$    & $(U,-)$    & $(U,-)~~~$ & $(U,-)$                 \\
			& & & & & \\
			$W$                        & $S$        & $T$        & $(U,-)~~~$ & $W$        & $W$                     \\
			& & & & & \\
			$(R,-)$                    & $S$        & $T$        & $(U,-)$    & $W$        & $(R,-)$
		\end{tabular}
	\end{table}
	\FloatBarrier	
	
	Let us identify the definable subcategories of $R$-$\mathrm{Mod}$ for $R=k[\epsilon:\epsilon^2=0]$. Recall that a module $M \in R \hbox{-}\mathrm{Mod}$ has form $M=R^{(\kappa)} \oplus U^{(\lambda)}$ for some cardinals $\kappa$ and $\lambda$, (see \cite{Per11}, Section 6.8). If both $\kappa$ and $\lambda$ are non-zero, then $\big{<}M\big{>}=R\hbox{-}\mathrm{Mod}$. Therefore, the only non-trivial proper definable subcategories of $R$-$\mathrm{Mod}$ are $\big{<}R\big{>}=\{R^{({\kappa})}:{\kappa} \mathrm{~a~cardinal}\}$ and $\big{<}U\big{>}=\{U^{(\lambda)}:\lambda \mathrm{~a~cardinal}\}$. 
	
	Since $\otimes_R$ commutes with direct sums it is easy to see that both $\big{<}R\big{>}^{\mathrm{def}}$ and $\big{<}U\big{>}^{\mathrm{def}}$ are closed under tensor product and $\big{<}U\big{>}^{\mathrm{def}}$ is even a tensor-ideal in $R \hbox{-}\mathrm{Mod}$. Furthermore, we have $\mathrm{hom}(R,-)=\mathrm{Hom}_R(R,-) \cong \mathrm{Id}_{R\hbox{-}\mathrm{Mod}}$ and therefore $\mathrm{hom}(R,U) \cong U$. It can also be checked that $\mathrm{hom}(U,U) \cong U$. Any object in $\big{<} U \big{>}^{\mathrm{def}}$ can be written as a direct limit of finite powers of $U$ and for any $N \in R \hbox{-}\mathrm{mod}$, $\mathrm{hom}(N,-)$ commutes with direct limits. Therefore, since $\mathrm{hom}(-,-)$ commutes with finite direct sums in both variables, $\mathrm{hom}(R,U) \cong U$ and $\mathrm{hom}(U,U) \cong U$ is enough to imply that $\big{<} U \big{>}^{\mathrm{def}}$ is fp-hom-closed. On the other hand, $\mathrm{hom}(U,R) \cong U$ meaning $\big{<} R \big{>}^{\mathrm{def}}$ is not fp-hom-closed.
	
	Next let us consider the corresponding Serre subcategories. First take $\mathcal{D}=\big{<}U\big{>}^{\mathrm{def}}$. Then
	
	$$S(U)=\epsilon U=0,$$
	
	$$T(U)=\mathrm{ann}_{U}(\epsilon)/\epsilon U=U/0 \cong U,$$
	
	$$(U,-)(U)=\mathrm{ann}_{U}(\epsilon)=U,$$
	
	$$W(U)=U/\epsilon U=U/0 \cong U$$ and
	
	$$(R,-)(U)=U.$$
	Therefore $\mathsf{S}_{\mathcal{D}}$ is generated by the indecomposable functor $S$ and indeed consists just of direct sums of copies of $S$. As $\mathrm{pdim}(S)=1$, by Proposition \ref{pdim}, $G \otimes S \in S_{\mathcal{D}}$ for every finitely presented $G:R\hbox{-}\mathrm{mod} \to \mathbf{Ab}$. Therefore, $\mathsf{S}_{\mathcal{D}}$ is a Serre tensor-ideal. 
	
	Now take $\mathcal{D}=\big{<}R\big{>}^{\mathrm{def}}$. Then
	
	$$S(R)=\epsilon R=U,$$
	
	$$T(R)=\mathrm{ann}_{R}(\epsilon)/\epsilon R=U/U \cong 0,$$
	
	$$(U,-)(R)=\mathrm{ann}_{R}(\epsilon)=U,$$

	$$W(R)=R/\epsilon R=U$$ and
	
	$$(R,-)(R)=R.$$
	
	Therefore $\mathsf{S}_{\mathcal{D}}$ is generated by the indecomposable functor $T$. As $T \otimes T \cong (U,-)$ this Serre subcategory is not closed under tensor product.
	
	In summary we get the following table, where $\big{<}~\big{>}^{\mathrm{def}}$ denotes `the definable subcategory generated by' and $\big{<}~\big{>}^{\mathsf{S}}$ denotes `the Serre subcategory generated by'.
	
	\begin{table}[htbp!]
		\begin{tabular}{l|lll|l|ll}
			Definable & Monoidal & fp-hom- & Tensor- & Serre  & Monoidal & Tensor-  \\
			subcat. & subcat. & closed & ideal & subcat. & subcat. & ideal \\ \hline	& & & & & & \\
			0  & Yes & Yes &  Yes & $(R\hbox{-}\mathrm{mod},\mathbf{Ab})^{\mathrm{fp}}$ & Yes & Yes \\
			& & & & & & \\
			$\big{<}U\big{>}^{\mathrm{def}}$  & Yes & Yes & Yes & $\big{<}S\big{>}^{\mathsf{S}}$ & Yes  & Yes  \\
			& & & & & & \\
			$\big{<}R\big{>}^{\mathrm{def}}$ & Yes & No & No & $\big{<}T\big{>}^{\mathsf{S}}$ & No & No          \\
			& & & & & & \\
			$R\hbox{-}\mathrm{Mod}$ & Yes & Yes & Yes & 0 & Yes & Yes \\
		\end{tabular}
	\end{table}
\end{example}

\subsubsection{Von Neumann regular rings}
	
Let us consider the example of von Neumann regular rings. 
	
	\begin{definition}
		A ring $R$ is von Neumann regular if for every $x \in R$ there exists some $y \in R$ such that $x=xyx$.
	\end{definition}

\begin{prop}
	Let $R$ be a commutative von Neumann regular ring so the normal tensor product of rings, $\otimes_R$, is a symmetric closed monoidal structure on $R\hbox{-}\mathrm{Mod}$. Every definable subcategory of $R\hbox{-}\mathrm{Mod}$ is fp-hom-closed.
\end{prop}

\begin{proof} \label{prop von Neumann =fp-hom}
	By (\cite{BRB}, Proposition 10.2.20), the global dimension of $(R\hbox{-}\mathrm{mod},\mathbf{Ab})^{\mathrm{fp}}$ is zero if and only if $R$ is von Neumann regular. Thus by Lemma \ref{pdim}, for $R$ von Neumann regular, every Serre subcategory of $(R\hbox{-}\mathrm{mod},\mathbf{Ab})^{\mathrm{fp}}$ is a tensor-ideal and therefore, by Theorem \ref{thm}, every definable subcategory is fp-hom-closed.	
\end{proof}

\begin{prop} \label{prop von Neumann =exactness}
	Let $R$ be a commutative von Neumann regular ring so the normal tensor product of rings, $\otimes_R$, is a symmetric closed monoidal structure on $R\hbox{-}\mathrm{Mod}$. Every fp-hom-closed definable subcategory $\mathcal{D}$ of $R\hbox{-}\mathrm{Mod}$ satisfies the exactness criterion. 
\end{prop}

\begin{proof}
$R$ is von Neumann regular if and only if every (left) $R$-module is flat, that is for every $M \in R\hbox{-}\mathrm{Mod}$, $M \otimes_R -:R\hbox{-}\mathrm{Mod} \to \mathbf{Ab}$ is exact (e.g. see \cite{BRB}, Theorem 2.3.22). Therefore, since $R$ is commutative, we obtain a symmetric closed monoidal product on $R\hbox{-}\mathrm{Mod}$ which is exact in each variable. Furthermore, by (\cite{BRB}, Proposition 10.2.38) we have $(R\hbox{-}\mathrm{mod},\mathbf{Ab})^{\mathrm{fp}}\simeq (R\hbox{-}\mathrm{mod})^{\mathrm{op}} \simeq R\hbox{-}\mathrm{mod}$ where the direction $R\hbox{-}\mathrm{mod} \to (R\hbox{-}\mathrm{mod},\mathbf{Ab})^{\mathrm{fp}}$ is given by the Yoneda embedding. Therefore, this equivalence is monoidal with respect to Day convolution product. In other words, letting $\mathcal{D}=R\hbox{-}\mathrm{Mod}$, $\mathrm{fun}(\mathcal{D})=(R\hbox{-}\mathrm{mod},\mathbf{Ab})^{\mathrm{fp}}\simeq R\hbox{-}\mathrm{mod}$ has an additive symmetric monoidal structure which is exact in each variable. Thus, by Proposition \ref{Lem exactness if}, $\mathcal{D}=R\hbox{-}\mathrm{Mod}$ satisfies the exactness criterion. Consequently any fp-hom-closed definable subcategory of $R\hbox{-}\mathrm{Mod}$ also satisfies the exactness criterion.	
\end{proof}

\begin{remark}
 By Proposition \ref{prop von Neumann =fp-hom}, for $R$ von Neumann regular, $\mathrm{Zg}(R\hbox{-}\mathrm{Mod})$ and $\mathrm{Zg}^{\mathrm{hom}}(R\hbox{-}\mathrm{Mod})$ are the same topology. Furthermore, by Proposition \ref{prop von Neumann =exactness} and Proposition \ref{prop exactness only if}, for every definable subcategory $\mathcal{D} \subseteq R\hbox{-}\mathrm{Mod}$, we can induce an exact, additive, closed, symmetric monoidal structure on the corresponding functor category $\mathrm{fun}(\mathcal{D})$. 
 
 By (\cite{BRB}, Proposition 3.4.30) the definable subcategory generated by $M \in R\hbox{-}\mathrm{Mod}$ is given by $(R/ \mathrm{ann}_R(M))\hbox{-}\mathrm{Mod}$ viewed as a full subcategory of $R\hbox{-}\mathrm{Mod}$ via $R \to R/ \mathrm{ann}_R(M)$. Therefore, it is easy to see directly that $\big{<}M\big{>}^{\mathrm{def}}$ is fp-hom-closed. Futhermore, the associated Serre subcategory of $(R\hbox{-}\mathrm{mod},\mathbf{Ab})^{\mathrm{fp}} \simeq R\hbox{-}\mathrm{mod}$ is given by $\{X \in R\hbox{-}\mathrm{mod}: (X,M)=0\}$. 
\end{remark}

\subsubsection{Coherent rings}

Now we consider the case where $R$ is coherent.

\begin{definition}
	A commutative ring $R$ is coherent if every finitely generated ideal is finitely presented. 
\end{definition}    

We will use the following properties.

\begin{prop} (\cite{BRB}, Corollary 2.3.18 and Proposition E.1.47)
	A (commutative) ring $R$ is coherent if and only if $R\hbox{-}\mathrm{mod}$ is abelian.
\end{prop}

\begin{prop} (\cite{BRB}, Theorem 3.4.24)
	A (commutative) ring $R$ is coherent if and only if the subcategory $\mathrm{Abs}\hbox{-}R\subseteq \mathrm{Mod}\hbox{-}R$ of absolutely pure modules is definable.
\end{prop}

\begin{lemma} \label{lem abs pure}
	Let $\mathcal{C}$ be as in Assumption \ref{Rmk C} and suppose $X \in \mathcal{C}$ and $U \in \mathcal{C}^{\mathrm{fp}}$ are such that $\mathrm{hom}(U,X)$ is absolutely pure. Let $f:A \to B$ be a morphism in $\mathcal{C}^{\mathrm{fp}}$. If $f:A \to B$ is a monomorphism in $\mathcal{C}$, then every morphism $h:A \otimes U \to X$ factors through $f \otimes U$. 
\end{lemma}

\begin{proof}
	Via the tensor-hom adjunction there exists some $h':B \otimes U \to X$ such that $h=h' \circ (f \otimes U)$ if and only if there exists some $\widehat{h'}:B \to \mathrm{hom}(U,X)$ such that $\widehat{h}=\widehat{h'} \circ f$ where $\widehat{h}:A \to \mathrm{hom}(U,X)$ is the morphism corresponding to $h$ via the adjunction isomorphism $(A \otimes U, X) \cong (A, \mathrm{hom}(U,X))$. But $\mathrm{hom}(U,X)$ is absolutely pure and $f:A \to B$ is a monomorphism with $A, B \in \mathcal{C}^{\mathrm{fp}}$ so applying (\cite{BRB}, Proposition 2.3.1) we get $\widehat{h}$ factors via $f$ as required.
\end{proof}

\begin{prop}
	Let $R$ be a commutative coherent ring. Any fp-hom-closed definable subcategory $\mathcal{D}$ of $\mathrm{Abs}\hbox{-}R$ satisfies the exactness criterion. 
\end{prop}

\begin{proof}
	Suppose $f:A \to B$ and $g:U \to V$ are morphisms in $\mathrm{mod}\hbox{-}R$ and $X \in \mathcal{D}$. Suppose further that $h:A \otimes U \to X$ satisfies $h=h' \circ (f \otimes U)=h'' \circ (A \otimes g)$ for some $h':B \otimes U \to X$ and $h'':A \otimes V \to X$, that is the following diagram commutes.
	
		\begin{tikzpicture}
	\matrix (m) [matrix of math nodes,row sep=4em,column sep=4em,minimum width=2em]
	{
		A \otimes U & B \otimes U   \\
		A \otimes V & X \\};
	\path[-stealth]
	(m-1-1) edge node [above] {$f \otimes U$} (m-1-2)
	(m-1-1) edge node [left] {$A \otimes g$} (m-2-1)
	(m-1-2) edge node [left] {$h'$} (m-2-2)
	(m-2-1) edge node [below] {$h''$} (m-2-2);
	\end{tikzpicture}
	
	As $\mathrm{mod}\hbox{-}R$ is abelian, we have exact sequences in $\mathrm{mod}\hbox{-}R$, $0 \to A' \xrightarrow{k} A \xrightarrow{f} B$ and $U \xrightarrow{g} V \xrightarrow{c} W \to 0$ where $k:A' \to A$ is the kernel of $f$ and $c:V \to W$ is the cokernel of $g$. Furthermore, since $A' \otimes -: \mathrm{mod}\hbox{-}R \to \mathrm{mod}\hbox{-}R$ is right exact we have an exact sequence \[A' \otimes U \xrightarrow{A' \otimes g} A' \otimes V \xrightarrow{A' \otimes c} A' \otimes W \to 0.\] 
	
	Now $h'' \circ  (k \otimes V) \circ  (A' \otimes g)= h'' \circ (A \otimes g) \circ (k \otimes U)=h' \circ (f \otimes U) \circ (k \otimes U)=0$. Therefore, $h'' \circ  (k \otimes V)$ factors via $A' \otimes c$, say $h'' \circ  (k \otimes V)=l \circ (A' \otimes c)$ for some $l:A' \otimes W \to X$. 
	
	\begin{tikzpicture}
	\matrix (m) [matrix of math nodes,row sep=4em,column sep=4em,minimum width=2em]
	{
		A' \otimes U & A' \otimes V & A' \otimes W & 0   \\
	&	A \otimes V & A \otimes W & \\
& X & &\\};
	\path[-stealth]
	(m-1-1) edge node [above] {$A' \otimes g$} (m-1-2)
	(m-1-2) edge node [above] {$A' \otimes c$} (m-1-3)
	(m-1-3) edge (m-1-4)
	(m-1-2) edge node [left] {$k \otimes V$} (m-2-2)
	(m-2-2) edge node [left] {$h''$} (m-3-2)
	(m-1-3) edge node [right] {$k \otimes W$} (m-2-3)
	(m-2-3) edge node [right] {$l'$} (m-3-2)
	(m-1-3) edge node [left] {$l$} (m-3-2);
	\end{tikzpicture}
	
	Note that $\mathcal{D} \subseteq \mathrm{Abs}\hbox{-}R$ is fp-hom-closed so $\mathrm{hom}(W,X)$ is absolutely pure. In addition, $k:A' \to A$ is a monomorphism therefore applying Lemma \ref{lem abs pure}, $l:A' \otimes W \to X$ factors via $k \otimes W$, say $l=l' \circ (k \otimes W)$, where $l':A \otimes W \to X$.  
	
	We have $h'' \circ (k \otimes V)=l' \circ (A \otimes c) \circ (k \otimes V)$. Setting $r=h'' - l' \circ (A \otimes c):A \otimes V \to X$ we have $r \circ (k \otimes V)=0$. Let $\hat{r}:A \to \mathrm{hom}(V,X)$ be the morphism corresponding to $r:A \otimes V \to X$ via the adjunction isomorphism. Then $\hat{r} \circ k=0$. 
	
	Now recall that $\mathrm{mod}\hbox{-}R$ is abelian and we have exact sequence $0 \to A' \xrightarrow{k} A \xrightarrow{f} B$. Therefore, $\mathrm{coker}(k)=\mathrm{im}(f)$. Write $f:A \to B$ as $i_f \circ \pi_f$ where $\pi_f:A \to \mathrm{im}(f)$ is the cokernel of $k$ and $i_f:\mathrm{im}(f) \to B$ is a monomorphism. Then $\hat{r}$ factors via $\pi_f:A \to \mathrm{im}(f)$, or equivalently, $r$ factors via $\pi_f \otimes V$, say $r=r' \circ (\pi_f \otimes V)$. Noting that $\mathrm{hom}(V,X)$ is absolutely pure, we may apply Lemma \ref{lem abs pure} to get that $r'=r'' \circ (i_f \otimes V)$ that is $r$ factors via $f \otimes V$. 
	
	Finally note that $h=h'' \circ (A \otimes g) =r \circ (A \otimes g) =r'' \circ (f \otimes V) \circ (A \otimes g)=r'' \circ (f \otimes g)$. Therefore we have shown that $h$ factors via $f \otimes g$ and the exactness criterion holds for $\mathcal{D}$.   	
\end{proof}

\subsection{Examples satisfying the rigidity condition}

Example \ref{Hopf} below gives a class of examples where the assumptions of Corollary \ref{Cfp rigid case} are satisfied.

\begin{example} \label{Hopf}
	The category of left $H$-modules, $H\hbox{-}\mathrm{Mod}$, for some Hopf algebra $H$ (e.g. a group algebra), has a closed symmetric monoidal structure by (\cite{Boh18}, Section 4.9). Furthermore, the finitely presentable left $H$-modules, $H\hbox{-}\mathrm{mod}$, form a symmetric rigid monoidal subcategory (see \cite{CGW06}, Section 4.1).
	
	Therefore, applying Corollary \ref{Cfp rigid case}, for every Hopf algebra $H$, the definable tensor-ideals of $H\hbox{-}\mathrm{Mod}$ correspond bijectively with the Serre tensor-ideals of $(H\hbox{-}\mathrm{mod},\mathbf{Ab})^{\mathrm{fp}}$.
\end{example} 

In particular let us consider an example from (\cite{P18}, Section 13).
 
 \begin{example} (\cite{P18}, Section 13) \label{Mike otimesk example}
 	Consider $R=k[\epsilon: \epsilon^2=0]$ as in Examples \ref{Mike otimesR example} but suppose further that the field $k$ has characteristic $2$. Then $R$ is a group ring. Indeed if we set $\epsilon+1=g$ and let $G=\big{<}g:g^2=1\big{>} \cong C_2$, then it is easy to see that $R \cong kG$ as rings. We can define a new tensor product $\otimes:R\hbox{-}\mathrm{Mod} \times R\hbox{-}\mathrm{Mod} \to R\hbox{-}\mathrm{Mod}$ given by $M \otimes N=M \otimes_k N$ and where the action of $R$ is determined by $g(M \otimes N)=gM \otimes gN$.
 	
 	Note that here the unit object is given by $U$ and the tensor product satisfies $R \otimes_k R \cong R^2$. We will use the notation of the previous example. The table below shows how this tensor product extends to $(R\hbox{-}\mathrm{mod},\mathbf{Ab})^{\mathrm{fp}}$. (See Section 13.5 of \cite{P18} for details of the calculation.)
 	
 	\FloatBarrier	
 	\begin{table}[htbp!]
 		\begin{tabular}{l|lllll}
 			$\otimes$               & $S$        & $T~~~$ & $(U,-)$    & $W$        & $(R,-)$          \\ \hline
 				& & & & & \\
 			$S$                     & $W$        &  0     & $S$        & $W$        & $(R,-)$          \\
 			& & & & & \\
 			$T$                     & 0          & $T$    & $T$        & 0          & 0                \\
 			& & & & & \\
 			$(U,-)$                 & $S$        & $T$    & $(U,-)~~~$ & $W$        & $(R,-)$          \\
 			& & & & & \\
 			$W$                     & $W$        & 0      & $W$        & $W$        & $(R,-)$          \\
 			& & & & & \\
 			$(R,-)$                 & $(R,-)~~~$ & 0      & $(R,-)$    & $(R,-)~~~$ & $(R,-)^2$
 		\end{tabular}
 	\end{table}
 	\FloatBarrier

We get the following definable subcategory/Serre subcategory correspondence, where as required by Corollary \ref{Cfp rigid case}, there is a one-to-one correspondence between the definable tensor-ideals of $R\hbox{-}\mathrm{Mod}$ and the Serre tensor-ideals of $(R\hbox{-}\mathrm{mod},\mathbf{Ab})^{\mathrm{fp}}$.
	
	\FloatBarrier
	\begin{table}[h!]
		\begin{tabular}{l|ll|l|ll}
			Definable & Monoidal & Tensor- & Serre  & Monoidal & Tensor-  \\
			subcategory & subcategory & ideal & subcategory & subcategory & ideal \\ \hline
			& & & & & \\
			0  & Yes &  Yes & $(R\hbox{-}\mathrm{mod},\mathbf{Ab})^{\mathrm{fp}}$ & Yes & Yes \\
			& & & & & \\
			$\big{<}U\big{>}^{\mathrm{def}}$  & Yes & No & $\big{<}S\big{>}^{\mathsf{S}}$ & No  & No  \\
			& & & & & \\
			$\big{<}R\big{>}^{\mathrm{def}}$ & Yes & Yes & $\big{<}T\big{>}^{\mathsf{S}}$ & Yes & Yes          \\
			& & & & & \\
			$R\hbox{-}\mathrm{Mod}$ & Yes & Yes & 0 & Yes & Yes \\
		\end{tabular}
	\end{table}
	\FloatBarrier
\end{example}

\end{document}